\documentclass[10.5pt]{article}

\usepackage[utf8x]{inputenc}
\usepackage[T1]{fontenc}
\usepackage[french,english]{babel}
\usepackage{url}

\usepackage{enumitem}
\usepackage{amsmath}
\usepackage{amsfonts}
\usepackage{mathrsfs}

\usepackage{bm}

\usepackage{amssymb}
\usepackage{amsthm}
\usepackage{bbm}

\usepackage{graphicx}
\usepackage{color}
\usepackage{geometry}
\usepackage{tikz}
\RequirePackage[colorlinks,citecolor=blue,urlcolor=blue]{hyperref}
\newcommand{\PP}{\mathbb{P}}

\newcommand{\E}{\mathbb{E}}

\newcommand{\R}{\mathbb{R}}
\newcommand{\N}{\mathbb{N}}

\renewcommand{\hat}{\widehat}
\renewcommand{\phi}{\varphi}
\newcommand{\rmd}{\mbox{d}}

\newcommand{\norm}[1]{\left\Vert #1 \right\Vert}
\newcommand{\rond}[1]{\mathscr{#1}}
\newcommand{\EE}[2][]{\E_{#1}\left[#2\right]}
\newcommand{\var}[2][]{\operatorname{Var}_{#1}\left(#2\right)}

\newcommand{\units}[1]{\mathbf{1}_{#1}}
\newcommand{\eps}{\varepsilon}
\newcommand{\ep}{\varepsilon}

\newcommand{\rk}{\mathfrak{r}} 
\newcommand{\pw}{{\rm pw}} 
\newcommand{\x}{\mathbf{x}} 
\newcommand{\cons}{\kappa}

\usepackage{color}

\definecolor{C}{RGB}{ 178, 41, 1 }

\newtheorem{thm}{Theorem}
\newtheorem{lemma}{Lemma}
\newtheorem{cor}{Corollary}
\newtheorem{rmk}{Remark}
\newtheorem{defi}{Definition}
\newtheorem{exe}{Example}

\begin{document}
\title{Adaptation to inhomogeneous smoothness  for  densities with irregularities}

\author{Céline Duval  \footnote{Universit\'e de Lille, CNRS, UMR 8524 - Laboratoire Paul Painlev\'e, F-59000 Lille, France.},
Émeline Schmisser\footnote{Universit\'e de Lille, CNRS, UMR 8524 - Laboratoire Paul Painlev\'e, F-59000 Lille, France.}}

\date{}

\maketitle

\begin{abstract} We estimate on a compact interval densities with isolated irregularities, such as discontinuities or discontinuities in some derivatives. From independent and identically distributed observations we construct a kernel estimator with non-constant bandwidth, in particular in the vicinity of irregularities. It attains  faster rates, for the risk $L_{p},\ p\ge 1$, than usual estimators with a fixed global bandwidth.  Optimality of the rate found  is established by a lower bound result. We then propose an adaptive method inspired by Lepski's method for automatically selecting the variable  bandwidth, without any knowledge of the regularity of the density nor of the points where the regularity breaks down. The procedure is illustrated numerically on examples.
 \end{abstract}

%

	\section{Introduction}
	
	\subsection{Context and motivations}

It has been observed in kernel non-parametric estimation that considering a variable bandwidth $h$ depending on the local behavior of the estimated function and on the dataset improves the rate of convergence. 
Such studies date back to \cite{breiman1977variable} which provides an extensive simulated study for multivariate density estimation and shows that  adding a term depending on the repartition of the observations improves the results. Afterwards, a number of options for the window were studied, most common choice of $h(x)$  involves nearest neighbor ideas or directly the estimated density  (see \textit{e.g.} \cite{abramson1982bandwidth,HM,HHM,TS}).  In  \cite{fan1992variable} Fan and Gijbels  consider a regression problem, the bandwidth is selected proportional to $h_{n}\alpha(X)$ where $h_{n}$ depends on the sample size and $\alpha$ is a function proportional to $f_{X}^{1/5}$ where $f_{X}$ is the density of the design. 
In all these references, the estimated densities have a fixed and known global regularity over the estimation interval.\medskip

Choosing a variable bandwidth is all the more important when estimating functions with varying regularity over the estimation interval, as studied in the present article. Indeed, if no special treatment is carried out to adapt locally to the regularity, a global choice of window is made, and the resulting estimator will consider that the irregularity represents the typical behavior of the function even if it is a singularity point. 

Studies have been conducted in specific cases. For instance, \cite{schuster1985incorporating} proposes a symmetrization device to improve the rate of convergence when the locations of the discontinuities are known.  In \cite{Eeden}, densities with non continuous derivatives or that are non continuous are  considered  (see also \cite{Gosh}).  She shows that the square $L_{2}$ risk attains  respectively  the rate $n^{-3/4}$ in the first case and $n^{-1/2}$ if the density is not $C^{0}$. The work \cite{CH} extends \cite{Eeden} and shows how the rate of convergence of a function that is $m$ times differentiable with  the derivative of order $m$ having simple discontinuities can be improved to obtain the rate of a $m+\tfrac12$ regular function. A similar context is examined in  \cite{Ka}, where  the regularity of the function is defined through its local behavior, or in \cite{van1997note} which proposes a central limit theorem for the integrated $L_{2}$ risk.    A numerical method is introduced in  \cite{DGL} to estimate densities that are  possibly discontinuous in an unknown finite number of points. It consists in viewing the estimation  problem as a regression problem, no theoretical study of the method is conducted. 

 The problem of the adaptive choice of window is a well studied question and a local bandwidth selection rule has been introduced in Lepski \textit{et al.} (1997) \cite{lepski1997optimal} in a regression framework (see also \cite{goldenshluger2008universal,MR3230001}). However the associated rates studied in these articles refer to a global regularity parameter which does not allow to fully appreciate the interest of adopting a local window instead of a global one.\medskip

A closely related issue is the study of bias correction at the boundary of the support of a compactly supported density. It has been studied in \cite{CFM} (see also the references therein and  \cite{DGL}) where a minimax and numerically feasible method based on local polynomial smoother is proposed. The support of the estimated function is known and the function has no irregularities in the interior of its definition domain. Some works also study the related but nonetheless separate problem of estimating the discontinuity points of the estimated function (see \textit{e.g.}  \cite{couallier1999estimation}).

\bigskip

In this work, focusing on kernel density estimators, we aim to extend the previous studies by defining classes of functions with isolated irregularities. Informally, we focus on functions that may be highly regular but have discontinuities (for the function itself or one of its derivatives, which may also be unbounded). This is the case for many common laws such as mixtures of uniform  or  Laplace distributions, which are infinitely differentiable except at certain points. Estimation of densities with compact support presenting irregularities at the boundaries when estimated over a wider interval is included in our framework. We highlight that a more refined study of the bias terms of the associated estimators and an appropriate choice of bandwidth can improve convergence rates.

  More precisely, if a function has a regularity $0<\alpha<\infty$ on a compact estimation interval $I_{0}$ except at certain points where a regularity $0\le\beta<\alpha$ is observed, a suitable choice of window $h_{0}$ (see Equation \eqref{eq:h0}) gives the following convergence rate (see Theorem \ref{thm:UBPH}) \[ \E\left[\int_{I_{0}}|\hat{f}_{h_0(x)}-f(x)|^p\rmd x\right] \leq C \left(n^{-\frac{p(\beta+1/p)}{2 \beta+1}}+\frac{\log(n)}{n}\units{p=2}+n^{-p\frac{\alpha}{2\alpha+1}}\right), \] where $ C$ is a positive constant. If $p\in[1,2]$ we find the rate associated with the greater of the two regularities and if $p>2$ we have a strict improvement in the exponent of the rate compared with the case where the function is seen to have overall $\beta$ regularity. A wide range of examples of classical laws are detailed to illustrate the improvement obtained compared to the classical case where the regularity is defined by a global parameter. To our knowledge, these results are novel in the literature. 

After studying the optimality of such  rate, we adapt the results of \cite{lepski1997optimal} to a density framework and obtain an adaptive inhomogeneous window selection method (see Corollary \ref{cor}). This method does not require any knowledge on the different regularities of the estimated density nor on the locations of the regularity breaks.

\bigskip

The paper is organised as follows. In the rest of this section, classical results on kernel density estimator are given as well as the associated rates of convergence, which will be a benchmark for our study. 

In Section \ref{sec:Princ} two classes of smoothness are introduced, over which we propose and study variable bandwidth estimators. 
The first is a class of piecewise H\"older functions for which considering an estimation window adapted to irregularities makes it possible in the case a  $L_{p},\ p\in[1,2]$ loss to recover the estimation rate of the irregularity-free function (see Theorem \ref{thm:UBPH}) and for a $L_{p},\ p>2$ loss to strictly improve the rate compared to the case where a fixed-window estimator is considered. The optimality of this rate is established in Theorem \ref{thm:LB}. This class of functions does not allow densities with unbounded derivatives, as  is the case for some Beta distributions. We introduce a second class of functions which are bounded but whose derivatives can be unbounded (see Theorem \ref{thm:UBDP}). An example of an unbounded density is also investigated in Example \ref{ex:5}. In Section \ref{sec:adapt}, an adaptive procedure is proposed and studied. 
It generalises to the density estimation problem the procedure introduced in Lepski \textit{et al.} (1997) \cite{lepski1997optimal}. It is illustrated numerically in Section \ref{sec:Num}. Finally, Section \ref{sec:prf} collects the different proofs.

\subsection{Kernel estimators}

	Consider $X_{1},\dots X_{n}$ i.i.d. (independent and identically distributed) random variables with density $f$ with respect to the Lebesgue measure. 
	We consider kernel density estimators, where the kernel satisfies the following Assumption.
	\begin{enumerate}[label={\textcolor{blue}{($\mathcal{A}$0)}}]
\item\label{Ass:K} Let $K$ be a symmetric bounded kernel  of order $\rk\in\N\setminus\{0\}$ with support in $[-1,1]$:  $\|K\|_{1}=1$ and for any $k, 1\leq k\leq \rk{-1},$ $\int x^k K(x)\rmd x=0$. 
\end{enumerate}
As $K$ is bounded with support in $[-1,1]$, the $L_p$ and the $L_2$-norms $\|K\|_{p}$ and $\|K\|_2$ are finite, and, for any $k>0$,  $\int |x|^k |K(x)|\rmd x<\infty$. 
	Assuming the kernel to be compactly supported simplifies the proofs, this constraint is also imposed in 
	 \cite{lepski1997optimal}. {As the kernel can be chosen, it is not a strong assumption}. 
Recall the definition of the classical kernel density estimator of $f$
\begin{align*}\hat f_h(x)=\frac 1n \sum_{i=1}^n K_h(X_i-x), \quad K_h(u)=\frac 1h K\left(\frac uh\right),\quad {{1\ge h}}> 0.\end{align*} The restriction to bandwidths $h$ smaller than $1$  is technical but natural as the bias term tends towards 0 with $h$. \\

\textbf{Notations}
Let us denote by $\star$ the convolution product, $u\star v(x)=\int u(t)v(x-t)\rmd t$ for functions $u, v$ such that the integral is well defined. For any function $w$, we set $w_h:=w\star K_h$. Denote by $\|.\|_{p,I}$, $p\ge 1$ and $I\subset\R$ an interval, the $L_{p}$-norm on $I$, namely $\|u\|_{p,I}:=\left(\int_{I}|u(x)|^{p}\rmd x\right)^{1/p}.$ Denote by $\lfloor \alpha\rfloor$  the integer \emph{strictly} less than the real number $\alpha$. In the sequel $C$ denotes a positive constant not depending on $n$ whose value may change from line to line, its dependancies are sometimes provided in indices. Finally, $a_{n}\lesssim b_{n}$ means that $a_{n}\leq Cb_{n}$.\\

It has been established (see \textit{e.g.} Tsybakov \cite{Tsyb}) using Rosenthal's inequality that when $f$ is bounded, for any $p\ge 1$ the following bias variance decomposition holds for the point-wise risk, for any $x\in\R$ and $h>0$,
\begin{align*}\E[|\hat f_{h}(x)-f(x)|^{p}]\le  C\left(|f(x)-f_{h}(x)|^{p}+\frac{1}{(nh)^{\frac p2}}\right),
\end{align*}
integrating this on $I_0$ a compact set, we get a bound for the integrated risk in
\begin{align*}\E[\|\hat f_{h}-f\|_{p, I_0}^{p}]\le  C'\left(\|f-f_{h}\|_{p,I_0}^{p}+\frac{1}{(nh)^{\frac p2}}\right)
\end{align*} for  positive constants $C,\ C'$ depending on $p$, $\|f\|_{\infty}$, $\|K\|_{2}$ and the length of $I_{0}$.

To derive the resulting rate of convergence one requires regularity assumptions on the function $f$. 
For instance if $f$ belongs to a Hölder ball  with regularity parameter $\alpha $ defined below, the punctual  and integrated bias terms can be bounded.

\begin{defi}[H\"older space]\label{holder}
A function $f$ supported on an interval $I\subset\R$ belongs to the H\"older space $\Sigma(\alpha, L,M,I)$ for $\alpha,L,M$ positive constants  if $f$ is in $C^{(\lfloor \alpha\rfloor)}(I)$, $\norm{f}_{\infty}\leq M$ and if for all $x,x'\in I$, $|x-x'|\leq 1$, it holds that $|f^{(\lfloor \alpha\rfloor)}(x)-f^{(\lfloor \alpha\rfloor)}(x')|\le L|x-x'|^{\alpha-\lfloor \alpha\rfloor}$.
\end{defi}
In general, the integrated risk is evaluated on Nikolski balls (see Tsybakov~(2009) \cite{Tsyb}, Definition 1.4 p.13), however for the classes of isolated irregularities we consider in the sequel, Hölder spaces are better adapted and provide sharper bounds, even when the bandwidth is fixed (see also \cite{Eeden}).

Therefore, if $f\in\Sigma(\alpha,L,M,I)$ with $\alpha$ greater than the order of the kernel $K$, for any $h<1$, the bias term $\|f_{h}-f\|^p$ is of order $h^{\alpha p}$.
 It follows  that selecting $h^{*}=n^{-\frac{1}{2\alpha+1}}$ leads to the minimax  rate $n^{-\frac{\alpha p}{2\alpha+1}}$ (see Hasminskii and Ibragimov \cite{10.1214/aos/1176347736}). This rule is relevant when the regularity of $f$ is well described by $\alpha$. Otherwise, we show that on the classes considered in the next Section, which allow densities with isolated irregularities, for the density itself or one of its derivatives, that a finer computation of the bias term enables to bypass the irregularity and provide faster rates of convergence for the integrated $L_{p}$-risk, $p\ge 1$. This motivates the use of a spatially adapted adaptive procedure as introduced in Section \ref{sec:adapt}.

	\section{Inhomogeneous bandwidth estimators\label{sec:Princ}}

\subsection{Rates on piecewise H\"older functions}

We introduce a class of regularity where the functions are piecewise Hölder such that on each sub-interval $]x_i, x_{i+1}[$ of the associated subdivision, the function is extendable on the right and on the left into a Hölder function on the interval $[x_i, x_{i+1}]$. A similar condition is introduced in  \cite{Eeden} (see condition B) for densities at most twice differentiable or in \cite{van1997note} (see Condition F therein).

\begin{defi}[Piecewise Hölder space]\label{def:PH}
Let $\alpha$, $L$, $M$ be positive, $I=]a,b[\subset \R$, $-\infty\le a<b\le \infty$, an interval and $\x=(x_1,\ldots,x_m)$ a vector such that $a<x_1<\ldots<x_m<b$ for some integer $m$.  Define $\Sigma_{{\rm pw}}(\x,\alpha,L,M,I)$ a set of functions $f$ such that, for any interval $[x_i,x_{i+1}]$, we can define a version $\bar f_i$ of $f$ that belongs to $\Sigma(\alpha,L,M,[x_i,x_{i+1}])$, for all $0\le i\le m$, where  $x_{0}= a $ and $x_{m+1}=b$.  If $a=-\infty$ (or $b=\infty$) then we impose instead that  $\bar{f}_{i}$ belongs to $\Sigma(\alpha,L,M,]-\infty,x_1])$ (or $\Sigma(\alpha,L,M,[x_m,\infty[)$) for $i=0$ (or $i=m$). 

Define the piecewise Hölder functions $\Sigma_{\pw}(\alpha,L,M,I)$ as: $f$ belongs to $\Sigma_{\pw}(\alpha,L,M,I)$ if there exists a vector $\x$ such that $f$ belongs to $\Sigma_{\pw}(\x,\alpha,L,M,I)$. 
\end{defi}

In other words, $f$ defined on $I$ belongs to $\Sigma_{\pw}(\alpha,L,M,I)$ if: $f$ is bounded by $M$ and there exists a vector $\x$ such that 
\begin{itemize}	
	\item $\forall x\notin\{ x_1,\ldots,x_{m}\}$, $f^{(\ell)}$ exists where $\ell=\lfloor \alpha\rfloor$, 
	\item $\forall  0\leq i\leq m$, $\forall x,x' \in [x_{i},x_{i+1}]$,  $|\bar f_i^{(\ell)}(x)-\bar f_i^{(\ell)}(x')|\leq L |x-x'|^{\alpha-\ell}$.
\end{itemize}
In particular it implies that $\bar f_{i}$ admits bounded derivatives up to order $\ell$ on each interval $[x_i,x_{i+1}],\ 0\le i\le m$. For example, the uniform distribution on $[0,1]$ is not in $C^0(\R)$, but belongs to $\Sigma_{\pw}(\ell,0,1,\R)$  for any integer $\ell$ (see also Example \ref{ex:1} below). The Laplace distribution  is continuous on $\R$, but its derivative is not continuous at 0:  it belongs to $\Sigma_{\pw}(\ell,\lambda^{\ell},\lambda,\R)$  for any integer $\ell$ (see also Example \ref{ex:2} below). \medskip

In the sequel, we consider densities in the space  $ \Sigma_{\pw}(\mathbf{x},\alpha,L,M,I)\cap \Sigma(\beta,L,M,I)$, this means that we suppose that at the irregularity points of $f$, the density has a regularity larger than $\beta\ge0$. Note that 
if the constant $M$ controlling the infinite norm of $f$ is indeed the same on both spaces, 
the constant $L$ could be chosen differently on  $ \Sigma_{\pw}(\mathbf{x},\alpha,L,M,I)$ and $ \Sigma(\beta,L,M,I)$. However to simplify notations we choose the same in both cases.\medskip

We estimate the density $f$ on a compact set $I_{0}$, the fact that it is bounded is essential here as we integrate the punctual risk to obtain a bound for the integrated one. For reasons of readability, we estimate the density on a compact interval $I_{0}$ smaller than the interval $I=]a,b[$ where the function is defined, $ -\infty\le a<b\le+\infty$.
Namely, we estimate $f$ on the compact interval $I_{0}=[c,d]\subset I$. Indeed, to construct our estimator, we need that for all bandwidth $1\ge h\ge 0$, it holds that $[c-h,d+h]\subset]a,b[$, which is satisfied if $[c,d]\subset[a+1,b-1]$. In practice, under Assumption \ref{Ass:K}, it is enough to impose  that $[c,d]\subset[a+n^{-1/(1+2\rk)}; b-n^{-1/(1+2\rk)}]$ as the largest bandwidth   is bounded by $n^{-1/(1+2\rk)}$. 
\begin{thm}\label{thm:UBPH}
Let $f$ be a density in $ \Sigma_{\pw}(\mathbf{x},\alpha,L,M,I)\cap \Sigma(\beta,L,M,I)$ for  in $I=]a,b[$ and positive constants  $0\le\beta<\alpha\leq \rk$, $L$ and $M$.    Under Assumption \ref{Ass:K},  define the bandwidth function $h_{0}$ on  $I_{0}\subset[a+1,b-1]$   by,  $\forall x\in I_{0}$
	\begin{align}
	\label{eq:h0}
	h_0(x)=\begin{cases}
		\left(\cons n\right)^{-\frac{1}{2\beta+1}}& \text{ if $\exists i\in\{1,\ldots, m\}$, } |x-x_{i}|\leq \left( \cons n\right)^{-\frac{1}{2\beta+1}},\\ 
		|x-x_{i}|& \text{ if  $\exists i\in\{1,\ldots, m\}$, } \left(\cons n\right)^{-\frac1{2\alpha+1}} \geq |x-x_{i}|> \left(\cons n\right)^{-\frac1{2\beta+1}}, \\
		\left(\cons n\right)^{-\frac1{2\alpha+1}} &\text{ elsewhere, }\end{cases} 
	\end{align}  
where  $\cons =\frac{  \norm{K}_1^2L^2}{\norm{K}_{2}^2M}$.
		Then,  for any $p\ge 1$
	it holds that 
	\[  \E\left[\|\hat{f}_{h_0}-f\|_{p,I_{0}}^p\right] \leq C  \left(n^{-\frac{p\beta+1}{1+2 \beta}}+\frac{\log(n)}{n}\units{p=2}+n^{-p\frac{\alpha}{2\alpha+1}}\right),\] where $ C$ is a positive constant  depending on $(p,\rk,L,M,\alpha,\beta,
	K)$.\end{thm}
It follows that the $L_{p}$-risk, $p\in[1,2]$,  of the estimator with variable bandwidth $\hat{f}_{h_0}$ is the same as if $f$ has regularity $\alpha$ on the whole interval $I$, namely as if $f$ were in $\Sigma(\alpha,L,M,I)$. Interestingly, this result implies that the rate of convergence reached by this estimator to recover a Gaussian density is the same (for the $L_{p}$-risk $p\in[1,2]$) for a uniform on $[0,1]$, a Laplace or an exponential density estimated on a  compact interval. 
If $p>2$ and if the order of the kernel $\rk$ is greater than $\alpha$, the rate of convergence is proportional to $$n^{-\frac{p\beta}{2\beta+1}\left(1+\frac1{\beta}\right)}\vee n^{-\frac{p\alpha}{2\alpha+1}},$$ which is equal to $n^{-p\alpha/(2\alpha+1)}$ only if the minimal regularity $\beta$ satisfies $\beta\geq \left(1-2/p\right)\alpha-1/p$. In Examples \ref{ex:1} and \ref{ex:2} below, we see that this is still faster than the rate for a fixed bandwidth. 

The idea behind the definition of the bandwidth function $h_{0}$ in \eqref{eq:h0} is the following. Consider to simplify the case where there is only one discontinuity in the regularity at $\mathbf x=0$. As the kernel $K$ is supported on $[-1,1]$, the estimator at point $x$ is computed from all the $X_{j}$ such that $X_{j}\in [x-h_{0}(x),x+h_{0}(x)]$. To avoid the discontinuity point 0 we choose $h_{0}(x)$ such that $0\notin ]x-h_{0}(x),x+h_{0}(x)[$, otherwise the bias is proportional to $h_0(x)^{\beta}$. If  $|x|\ge (\cons n) ^{-1/(2\alpha+1)}$ is sufficiently apart from 0, the chosen bandwidth is the same as for a fixed bandwidth for a function of regularity $\alpha$ (the bias and the variance terms have the same order).  
 If $|x|$ is smaller  than $ (\cons n)^{-1/(2\alpha+1)}$, the bandwidth $h(x)$ should not  be too large, such that 0 does not belong to the interval  $]x-h(x),x+h(x)[$.  The bias remains small, of order   $h(x)^{\alpha}$. The variance becomes greater than the bias term, so we choose the largest bandwidth such that 0 does not belong in $]x-h(x),x+h(x)[$, namely $h_0(x)=|x|$. 
		If $|x|$ is  too small, this solution does not work any longer, indeed, the variance term would be proportional to $1/(n|x|)$. So if $|x|$ is smaller than $(\cons n)^{-1/(1+2\beta)}$, we take the bandwidth adapted for the minimal regularity $\beta$, $(\cons n)^{-1/(1+2\beta)}$. We can note that if $\beta=0$,   the bias term no longer tends to 0: there is no point in controlling the variance, and it is possible to take any value for $h$ between $1/n$ and 1. The point-wise risk will have order 1, but on an interval of length $1/n$.  It will not affect the integrated $L_2$-risk, however, for large values of $p$, this term can be predominant in the risk. 
		
		 The constant $\kappa$ in the definition of $h_{0}$ does  not affect the rate of convergence and could be replaced by 1 or any other positive value. Nonetheless, when it comes to the adaptive result and Corollary \ref{cor} below we rely on the property that for this choice of $h_{0}$ (and therefore of $\kappa$) it holds that the square biais is smaller than the majorant of the variance considered in the adaptive procedure of Section \ref{sec:adapt}.

\begin{rmk}
It is possible to allow in $\Sigma_{\pw}(\mathbf x,\alpha, L,M,I)$  specific regularities in each intervals $[x_{i},x_{i+1}]$ by replacing the pairwise $(\alpha, L)$ by vectors $\pmb{\alpha}:=(\alpha_{0},\ldots,\alpha_{m})\in\R_{+}^{m+1}$, and $\pmb{L}:=(L_{0},\ldots,L_{m})$. However it would not improve the rate which is imposed by the interval with the smallest regularity. The upper bound in Theorem \ref{thm:UBPH} would be the same with $\alpha$ replaced by $\underline {\pmb{\alpha}}=\min_{0\le i\le m}\alpha_{i}.$ 
\end{rmk}

\begin{proof}[Proof of Theorem \ref{thm:UBPH}]
	For the sake of simplicity, the proof is done when there is only one isolated irregularity, at 0.  We first control the point-wise risk on the interval $ I_{0}$ containing 0. Let $p\ge 2$, we can apply the Rosenthal inequality to the i.i.d. centered variables $Z_{i}:=K_{h_{0}(x)}(X_{i}-x)-f_{h_{0}(x)}(x)$. For that note that $\var{Z_{i}}\le \E[K_{h_{0}(x)}(X_{i}-x)^{2}]\le M\|K\|_{2}^{2}h_{0}(x)^{-1}$ and that for $p\ge 2$ after a change of variable one gets	\begin{align*}
	\E[|Z_{i}|^{p}]&\le \frac{2^{p}}{h_{0}(x)^{p}} \int\left|K\left(\frac{y-x}{h_{0}(x)}\right)\right|^{p}f(y)\rmd y\le  \frac{2^{p}M\|K\|_{p}^{p}}{h_{0}(x)^{p-1}} .
	\end{align*} 
Then 
\[
		 \E{\left[|\hat{f}_{h_0(x)}(x)-f(x)|^p\right]}\leq 2^p\left(\ \E{\left[|\hat{f}_{h_0(x)}(x)-f_{h_{0}(x)}(x)|^p\right]}+\left|f_{h_0(x)}(x)-f(x)\right|^p\right)\\ 
\]
 and by the Rosenthal inequality we get the following bound for the variance term
	\[\EE{|\hat{f}_{h_0(x)}(x)-f_{h_0(x)}(x)|^p}\leq C_p
	\max\left\{ \frac{M^{p/2}\norm{K}_2^{p}}{(nh_0(x))^{p/2}},\frac{2^{p}M\norm{K}_p^p}{(nh_{0}(x))^{p-1}}
	\right\}. 
	\]
The dominant term is proportional to $\frac{M^{p/2}\norm{K}_2^{p}}{(nh_0(x))^{p/2}}$.
Now, if $p\in[1,2)$, we derive the result from the case $p=2$ using that $x\mapsto x^{p/2}$ is concave together with the Jensen inequality ensuring that $\E[|Z|^{p}]\le \big(\E[Z^{2}]\big)^{p/2}$.
It follows, for any $p\geq 1$, 
\begin{equation}
	\EE{|\hat{f}_{h_0(x)}(x)-f_{h_0(x)}(x)|^p}\leq c_p\left( v^p(h_0(x))+|f_{h_0(x)}(x)-f(x)|^p\right) \label{eq:Rosprf1}
\end{equation}
with $v^{2}(h)$ is a bound on the  variance term defined as  \begin{align}\label{eq:vh}v(h):=\frac{M^{1/2}\norm{K}_2}{(nh)^{1/2}}\end{align}
and $c_p$ is a constant depending only on $p$. 
Depending on the value of $x$ with respect to 0, we consider two different bounds for the bias term. Let us set $\ell=
\lfloor \alpha\rfloor$.	
	
		$\bullet$ 	For $|x|>(\kappa n)^{-1/(2\alpha+1)}$, it holds that $\forall x'\in[x-h_0(x),x+h_0(x)]$, $f^{(\ell)}(x')$ exists and $|f^{(\ell)}(x)-f^{(\ell)}(x')|\leq L |x-x'|^{\alpha}$. 
		Then, we get using a change of variable and  \ref{Ass:K}
		\begin{align}
			f_{h_0(x)}(x)-f(x)&= 
	\int_{-1}^{1} K(u)\big[f(x+uh_0(x)))-f(x)\big]\rmd u\nonumber\\
&=\int_{-1}^1 K(u)\left( uh_0(x) f'(x)+\ldots +\frac{u^\ell h_0(x)^\ell}{\ell!} f^{(\ell)}(x+c(u)h_0(x))\right)\rmd u\nonumber\\
			&= h_0(x)^\ell \int_{-1}^1 K(u)\frac{u^\ell }{\ell!} (f^{(\ell)}(x+c(u)h_0(x))-f^{(\ell)}(x)))\rmd u\label{eq:BiasTayl}
		\end{align}
		where we used Taylor-Lagrange formula, for some $c(u)\in[0\wedge u,0\vee u]$, and  \ref{Ass:K}.  Therefore, we derive the upper-bound on the bias term 
		\begin{align} |f_{h_0(x)}(x)-f(x)|&\leq  h_0(x)^{ \alpha} \int_{-1}^1|u^{\alpha}K(u)|\rmd u\frac{L}{\ell!}
		\leq 
	h_0(x)^{\alpha}\norm{K}_1L\label{eq:BiasPH1}.\end{align}
			Check that for $h_0=(\cons n)^{-1/(2\alpha+1)}$, it holds
	$|f_{h_0(x)}(x)-f(x)|\leq L\norm{K}_1h_0(x)^{\alpha}=v(h_0(x)).$

		 $\bullet$ If $|x|\leq (\kappa n)^{-1/(2\beta+1)}$, the density $f$ belongs to the Hölder space $\Sigma(\beta,L,M,I)$ and the bound for the bias is done similarly replacing $\alpha$ with $\beta$. We also have that 
		 $ |f_{h_0(x)}(x)-f(x)|\leq v(h_0(x)).$

	$\bullet$ Elsewhere, selecting $h_{0}(x)=|x|$ allows to remain sufficiently away from 0 and  Equation \eqref{eq:BiasPH1} still applies. The bound  $L\norm{K}_1 h^{\alpha}$ on the bias is equal to the bound on the variance  $v(h)$ for $h^{*}=(\kappa n)^{-1/(2\alpha+1)}$, here, $h_{0}$ is smaller than this $h^{*}$. Since the function $v(h)$ is decreasing with $h$ and the bound on the bias  is increasing, we obtain that $$|f_{h_0(x)}(x)-f(x)|\le |f_{h^{*}}(x)-f(x)|= v(h^{*})\le v(h_0(x)).$$ 
Therefore, for any $x\in I_0$, we have the property that
	\begin{equation}
		|f_{h_0(x)}(x)-f(x)|\leq v(h_0(x)).\label{eq:lienbiaisvariance}
\end{equation}	This inequality will be important for the adaptive result. 

Gathering all terms and integrating on the compact interval $I_{0}$ containing 0 one gets for $p\ge 1$
	\begin{align*}
			 \E\left[\|\hat{f}_{h_0}-f\|_{p,I_{0}}^p\right] &\leq 
		 C\Bigg\{\int_{0}^{(\kappa n)^{-\frac{1}{2\beta+1}}}\hspace{-0.7cm}n^{-\frac{p\beta}{2\beta+1}}\rmd x
		+\int_{(\kappa n)^{-\frac{1}{2\beta+1}}}^{(\kappa n)^{-\frac{1}{2\alpha+1}}} |x|^{p\alpha} +\frac{1}{(n|x|)^{\frac p2}}\ \rmd x+\int_{I_{0}} n^{-\frac{p\alpha}{2\alpha+1}}\rmd x\Bigg\}\\
		&\leq C \left(n^{-\frac{p\beta+1}{2\beta+1}}+\frac{\log(n)}{n}\mathbf{1}_{p=2}+n^{-\frac{p\alpha}{2\alpha+1}}\right), 
	\end{align*}
where $C$ depends on $(p,\rk,L,\alpha,\|K\|_{2},\|K\|_{p},\|K\|_{\infty},\|u^{\alpha}K\|_{1})$. Generalization to $m$ isolated regularities in $I_{0}$ is straightforward by cutting the integral around each irregularity point $x_{i},\ 1\le i\le m$.
	\end{proof}

\begin{exe} \label{ex:1}\textbf{Uniform distribution} Let   $f(x)=\frac{1}{b-a}\units{a\leq x\leq b}$ and $I_{0}=[a-1,b+1]$. Then, $f$ is not continuous on $\R$, but belongs to $\Sigma_{\pw}(\mathbf x =(a,b),\ell,0,1/(b-a),\R)\cap \Sigma(0,1/(b-a),1/(b-a),\R)$ for any integer $\ell$.
For a fixed bandwidth $h>0$ and choosing $\ell=\rk$, where $\rk$ is the order of the kernel, we have 
  \[ \int_{-1}^1|f_h(x)-f(x)|^{p}\rmd x\lesssim \int_{|x|\leq h} 1\rmd x+\int_{h\leq |x|\leq 1} h^{p\rk}\rmd x\lesssim h.\]
 It follows that the optimal fixed window for the $L_2$-risk is $h=n^{-1/2}$ leading to a risk in $n^{-1/2}$. This is exactly the rate found by \cite{van1997note}. Our  variable bandwidth estimator attains the faster rate $n^{-2\rk/(2\rk+1)}$. 
 Similarly, the $L_p$-risk of an estimator with fixed bandwidth  is $n^{-1/(1+2/p)}$. If $\rk\ge 1$, it can be improved in $n^{-1}\vee n^{-{\rk p}/({2\rk+1})}$ considering $\hat f_{h_{0}}$, and if $p>2$ and $\rk\geq 1/(p-2)$, then the achieved rate is $n^{-1}$. \end{exe} 
  
\begin{exe}\label{ex:2} \textbf{Symetrized exponential} Let $f(x)=\frac{\lambda}{2}e^{-\lambda|x|}$ which is  continuous on $\R$, but its derivative is not continuous at 0,  and $I_{0}=[-1,1]$. For the  fixed bandwidth estimator, a convergence of $n^{-3/4}$ for the $L_2$-risk is obtained in \cite{van1997note}. Note that $f$ also belongs to $\Sigma_{\pw}(\mathbf x =0,\ell,\lambda^{\ell+1}/2,\lambda/2,\R)\cap \Sigma(1,\lambda^2,\lambda/2,\R)$ for any integer $\ell$, we set $\ell=\rk$. Then, the $L_2$-risk of  $\hat{f}_{h_0}$ is smaller than $n^{-2\rk/(2\rk+1)}\ll n^{-3/4}$ whenever $\rk\ge 2$. 
For any $p\ge 1$, the bias for a fixed bandwidth $h>0$ is bounded by
  \[\int_{-1}^1|f_h(x)-f(x)|^pdx\lesssim \int_{|x|\leq h} h^p\rmd x+\int_{h\leq |x|\leq 1} h^{\rk p}dx\lesssim h^{p+1},\] then selecting $h=n^{-1/(3+2/p)}$ performs the squared bias-variance compromise leading to a rate proportional to $n^{-(p+1)/(3+2/p)}$. 
Considering $\hat f_{h_{0}}$, it can be improved to get the faster rate $n^{-(1+p)/3}\vee n^{-\rk p/({2\rk+1})}$.

\end{exe}

\subsection{Lower bound result}

To study the optimality of the rate implied by Theorem \ref{thm:UBPH} we prove a lower bound result on a slightly larger class. Define $H_{\alpha,\beta}(\eps)=\Sigma(\beta,L,M,I)\cap \Sigma^{\eps}_{pw}(\alpha,L,M,I)$, where $f\in  \Sigma^{\eps}_{pw}(\alpha,L,M,I)$ if there exists a vector $\mathbf x=(x_{1},\ldots, x_{m})$ such that for all $i$, $f$ lies in $\Sigma(\alpha, L,M,[x_{i}+\eps,x_{i+1}-\eps]).$ Contrary to the space considered in Theorem \ref{thm:UBPH}, $H_{\alpha,\beta}(\eps)$ allows functions to have a regularity $0\le \beta<\alpha$ on an $\eps$-interval around each irregularity point. Note that   $\Sigma(\beta,L,M,I)\cap\Sigma_{pw}(\alpha,L,M,I)\subseteq \lim_{\ep\to 0} H_{\alpha,\beta}(\ep).$ 

\begin{thm}\label{thm:LB}
Let $\alpha\ge 1$ and $0\le \beta<\alpha$,  consider some bounded interval $I_{0}\subset I,$ where $I \subset\R$ is an interval, and positive constants $L$ and $M$. Let $\eps>0$, then,  for any $p\ge 1$ 
	it holds that 
	\[ \inf_{\hat f}\sup_{f\in H_{\alpha,\beta}(\eps)} \E\left[\|\hat{f}-f\|_{p,I_{0}}^p\right] \geq c  \left(n^{-\frac{p\beta+1}{2 \beta+1}}\vee n^{-\frac{p\alpha}{2\alpha+1}}\right),\] where $ c>0$ is a  constant  depending on $(p,L,M,\alpha,\beta)$ but not on $\eps$. 
\end{thm}
The proof  is provided in Section \ref{sec:prf}.  In the proof observe that a possible value for $\eps$ is $n^{-1/(2\beta+1)}$.

\subsection{Classes allowing unbounded derivatives}

The case where derivatives of the density are unbounded are not covered by Theorem \ref{thm:UBPH} even though it encompasses classical examples such as some beta distributions. In this case also, the rate of classical kernel density estimators can be improved considering variable bandwidth estimators (see Examples \ref{ex:3}, \ref{ex:4} below and  \ref{ex:5} where the density itself is unbounded). Consider the following class of bounded densities whose derivatives may be unbounded.
\begin{defi}[Piecewise Differentiable space]\label{def:PHU}
For an integer $\ell$, positive constants $L$, $M$, $\gamma$,  an interval $I\subset \R$ and $\x=(x_1,\ldots,x_m)$ a vector in $I$, define  $D_{\pw}(\x,\ell,\gamma,L,M,I)$   a space of piecewise differentiable functions  as follows. A function $f$ belongs to $D_{\pw}(\x,\ell,\gamma,L,M,I)$ if  $\norm{f}_{\infty,I}\leq M$ and  for all $x\in I\setminus\{x_1,\ldots,x_m\}$, $f^{(\ell)}$ exists  and satisfies  $|f^{(\ell)}(x)|\leq L\left(\max_{1\le i\le m} |x-x_i|^{-\gamma}\vee 1\right)$. 

We define $D_{\pw}(\ell,\gamma,L,M,I)$ as: $f$ belongs to $D_{\pw}(\ell,\gamma,L,M,I)$ if there exists a vector $\x$ such that $f$ belongs to $D_{pw}(\x,\alpha,L,M,I)$. 
\end{defi}

Assume that $\ell\leq \rk$, indeed the specific conditions imposed on the derivative $f^{(\ell)}$  are only visible in the computation of the bias term if $\ell\leq \rk$. As in the previous result we estimate $f$ on a compact interval $I_{0}$.

\begin{thm}\label{thm:UBDP}
Consider a density $f\in D_{\pw}(\x,\ell,\gamma,L,M,I)\cap \Sigma(\beta,L,M,I)$ with $0\le \beta<\ell$. Set $c_0:=c_{0}(f)=\min\{1,\frac{1}{2}\underset{1\le i\le m}{\min}|x_{i+1}-x_i|\}$. Suppose Assumption \ref{Ass:K} with $\rk\geq \ell$ and consider a compact interval $I_{0}\subset I$. Set $\cons=L^22^{2\gamma}\norm{K}_1^2/(\norm{K}_2^2M)$.

i). If $\gamma\leq \ell-\beta$ define the bandwidth function $h$ on  $I_{0}$ by 
	\begin{align}
	\label{eq:hD1}
	h_0(x)=\begin{cases}
		\frac12 (\cons n)^{-\frac{1}{2\beta+1}}& \text{ if }\exists i\in\{1,\ldots, m\},  |x-x_{i}|\leq (\cons n)^{-\frac1{2\beta+1}},\\ 
		\frac12|x-x_{i}|& \text{ if } (\cons n)^{-\frac1{1+2\beta}}\leq |x-x_i|\leq (\cons n)^{-\frac1{2\ell-2\gamma+1}},\\
		\frac12(\cons n)^{-\frac1{2\ell+1}} |x-x_i|^{\frac{2\gamma}{2\ell+1}} & \text{ if } (\cons n)^{-\frac1{1+2\ell-2\gamma}} \leq |x-x_i|\leq c_0, \\
		\frac12(\cons n)^{-\frac{1}{2\ell+1} }&\text{ if } \forall i \in \{1,\ldots,m\}, |x-x_i|\geq c_0.\\
\end{cases} 
	\end{align}  
Then, for $p\ge 1$ and a positive constant $C$ ,
it holds that 
\[ 	 \E\left[\|\hat{f}_{h_0}-f\|_{p,I_{0}}^p\right] \leq C\left( n^{-\frac{p\beta+1}{2\beta+1}}+\frac{\log(n)}{n}\units{p=2}+n^{-\frac{p\ell}{2\ell+1}}\left(1+\log(n)\units{\gamma=\frac{2\ell+1}{p}}\right)\right).\]

ii).  If $\gamma>\ell-\beta$ define the bandwidth function $h$ on  $I_{0}$ by 
\begin{align}
	\label{eq:hD2}
	h_0(x)=\begin{cases}
		\frac12(\cons n)^{-\frac{1}{2\beta+1}}& \text{ if }\exists i\in\{1,\ldots, m\},  |x-x_{i}|\leq (\cons n)^{-\frac{(\ell-\beta)/\gamma}{2\beta+1}},\\ 
		\frac12(\cons n)^{-\frac1{2\ell+1}} |x-x_i|^{\frac{2\gamma}{2\ell+1}} & \text{ if } (\cons n)^{-\frac{(\ell-\beta)/\gamma}{2\beta+1}} \leq |x-x_i|\leq c_0 ,\\
		\frac12(\cons n)^{-\frac{1}{2\ell+1} }&\text{ if } \forall i \in \{1,\ldots,m\}, |x-x_i|\geq c_0.\\
	\end{cases} 
\end{align}  

		Then, for $p\ge 1$ and a positive constant $C$ ,
it holds that 
		\[  \E\left[\|\hat{f}_{h_0}-f\|_{p,I_{0}}^p\right] \le C\left(
			n^{-\frac{ p\beta+(\ell-\beta)/\gamma}{(2\beta+1)}}
			+n^{-\frac{p\ell}{2\ell+1}}\left(
				1+\log(n)\units{\gamma=\frac{2\ell+1}{p}}
		\right)		\right).		
	\] 
	
	\end{thm}
In Equations \eqref{eq:hD1} and \eqref{eq:hD2}, the definition of $c_{0}$ ensures that any point $x$ can be at distance smaller than $c_0$ of at most one irregularity point $x_i$, ensuring that both functions $h_0$ are well defined. 
Note that when the derivative is not too large, the rate of convergence is the same (up to a $\log(n)$ factor) as for bounded derivatives. Even when the derivatives are large, the rate of $\hat{f}_{h_{0}}$ improves the usual rate on $\Sigma(\beta,L,M,I)$.

The bandwidth $h_{0}$ is constructed in the same way as for piecewise Hölder spaces. If $x$ is close to one of the irregularities, we take into account  the minimum regularity $\beta$ (which may be 0). If $x$ is larger than $c_{0}$ and therefore outside the zone of influence of the irregularities, we take the usual bandwidth adapted to the $\ell$ regularity.  In the intermediate regime, we have to be more careful for the choice of the bandwidth. 
First, as the derivative is unbounded, we impose that $h_{0}(x)\leq \min_i|x-x_i|/2$ to ensure that $f^{(\ell)}$ can be bounded on the estimation interval $ [x-h_{0}(x),x+h_{0}(x)]$.
Secondly, we can remark that the bias can be bounded in two ways (see Equations \eqref{eq:BiasTayl}  and \eqref{eq:BIASup})
\begin{align}\label{eq:biasNB} |f_h(x)-f(x)|\lesssim \min\left\{ h^{\beta}, h^{\ell}|x-x_i|^{-\gamma}\right\}.
\end{align}
Depending on the respective values of $\gamma,\ \beta$ and $\ell$ this second bias term depending on $x$ induces additional regimes for the selection of the bandwidth $h_{0}$.

\begin{exe}\label{ex:3}\textbf{Bounded beta distributions}
Consider for $\alpha,\beta\geq 0$ the bounded density $f_{\alpha,\beta}(x)= c_{\alpha,\beta} x^{\alpha}(1-x)^{\beta}\mathbf{1}_{x\in[0,1]}$, where $c_{\alpha,\beta}=\Gamma(\alpha+1)\Gamma(\beta+1)/\Gamma(\alpha+\beta)$, which is $C^{\infty}$ on any interval not containing the points $\{0,1\}$. For any integer $\ell\ge\alpha\wedge\beta$ it holds that $f_{\alpha,\beta}^{(\ell)}(x)\leq C_{\alpha,\ell} x^{\alpha-\ell}+C_{\beta,\ell}(1-x)^{\beta-\ell}$ for $x\notin\{0,1\}$.
Select $\ell=\rk\ge\alpha\wedge\beta$, if the bandwidth $h>0$ is fixed, we have that 
	\[\int_{-\frac12}^{\frac12}|f_h(x)-f(x)|^pdx\lesssim \int_{|x|\leq 2h}h^{p\lfloor \alpha\rfloor}dx+\int_{2h\leq|x|\leq 1/2} h^{p\rk}|x|^{p(\alpha-\rk)}dx\lesssim h^{p\lfloor \alpha\rfloor +1}+h^{p\alpha+1} .\]
	In the same way, we can prove that $\int_{1/2}^{3/2}|f_h(x)-f(x)|^pdx\lesssim h^{p\lfloor \beta\rfloor +1}$. 
	Therefore, the bias term (at the power $p$) is bounded by $h^{p\lfloor \alpha\wedge\beta\rfloor +1}$. The rate of convergence is obtained for $h=n^{-1/(2(\lfloor \alpha\wedge\beta\rfloor +1+2/p))}$ and the the risk is bounded by $n^{-\frac{p\lfloor \alpha\wedge\beta\rfloor +1}{2\lfloor \alpha\wedge\beta\rfloor+1+2/p}}.
$
 Second, observe that $f_{\alpha,\beta}$ belongs to $D_{\pw}(\rk, \rk-(\alpha\wedge \beta),C_{\alpha,\beta}, c_{\alpha,\beta},\R)\cap \Sigma(\alpha\wedge \beta,C_{\alpha,\beta}, c_{\alpha,\beta},\R)$. Considering the variable bandwidth $h_{0}$ in \eqref{eq:hD1} and $I_{0}=[-2,2]$, Theorem  \ref{thm:UBDP} implies a $L_p$-risk of the estimator $\hat f_{h_{0}}$ in 
$n^{-\frac{p(\alpha\wedge\beta)+1}{2(\alpha\wedge \beta)+1}}+\frac{\log(n)}{n}\units{p=2}+n^{-\frac{p\rk}{2\rk+1}}$.
Comparing both rates, 
	we observe that the rate of  $\hat f_{h_{0}}$ is  systematically faster. 
\end{exe}

\begin{exe}\label{ex:4}Consider a case where $\ell-\beta <\gamma$, for example $f(x)=\frac14(1+\sin\big(\frac1x\big))\mathbf{1}_{[-2,2]}(x),$ which has a discontinuity point at 0. Set $I_{0}=[-1,1]$, it is straightforward to see that $|f^{(\ell)}(x)|\le C_{\ell}|x|^{-2\ell}$, $x\in I_{0}\setminus\{0\}$, for some constant $C_{\ell}$ and any $\ell$. Select $\ell=\rk$,
if the bandwidth $h$ is fixed, then the bias term is bounded by, for any $0\le \delta\leq 1$ 
	\[\int_{-1}^1|f_h(x)-f(x)|^pdx \lesssim \int_{|x|\leq h^{\delta}} 1dx+\int_{h^{\delta}\leq |x|\leq 1} |x^{-2\rk}h^{\rk}|^pdx\lesssim h^{\delta}+h^{-2\delta p\rk+\rk p+\delta}.\]
	This quantity is minimal when $\delta=1/2$, for which the bias is bounded by $\sqrt{h}$. Then the optimal window is proportional to $n^{-p/(1+p)}$ and the risk is bounded by $n^{-p/(2(1+p))}$. 
	The function $f$ belongs to $D_{\pw}(\rk, 2\rk,C_{\ell}, \frac14,]-2,2[)\cap \Sigma(0,1/2,1/2,]-2,2[)$ and that $\hat f_{h_{0}}$ with $h_{0}$ given by \eqref{eq:hD2}   provides a faster rate in $n^{-p\rk/(2\rk+1)}+n^{-1/2}$ for any any $p\ge 1$. 
\end{exe}

\begin{exe}\label{ex:5} \textbf{Unbounded densities}
Consider a density $f$ supported on $(0,\infty)$ and such that $f(x)\leq C (x^{-\eta}\vee 1)$ with $0<\eta<1/2$ and estimation interval $I_{0}=]0,A]$ for some $A>1$. Then $f$ is in  $L_2(I_0)$ and may tend to $\infty$ at 0. 
Assume that $f$ is in $C^{\ell}$ and that $|f^{(\ell)}(x)|\leq C x^{-\eta-\ell}$ for some integer $\ell>0$ and $x>0$. 
The variance of the estimateur $\hat{f}_h(x)$ is bounded by 
\[\var{\hat{f}_h(x)}\lesssim\min\left( \frac{1}{nh^2}, \frac{x^{-\eta}}{nh}\right),\quad \mbox{for any }h\ge 0,\ x\ge 0.\]
The squared bias term can be bounded as follows:
\[|f_{h}(x)-f(x)|^{2}\lesssim\begin{cases}
h^{2\ell}x^{-2(\eta+\ell)} & \text{ if } x\geq 2h,\\
x^{-2\eta} & \mbox{otherwise}.
\end{cases}\]
The first inequality is derived from the second majorant in \eqref{eq:biasNB} replacing $\eta$ with $\eta-\ell$, the second follows from the assumption on $f$ together with
$$f_{h}(x)=\int_{-\frac xh}^{1}f(x+uh)K(u)\rmd u\lesssim\int_{-\frac xh}^{0}(x+uh)^{-\eta}du+x^{-\eta}\lesssim
\frac{x^{1-\eta}}{h}+x^{-\eta},$$ and noting that for $h\ge x/2$ it holds $x^{1-\eta}/h\le 2x^{-\eta}$. The squared bias variance compromise leads to select $h$ as follows
$$h_{0}(x)=\begin{cases}
{\frac{x^{2\eta}}{n} }& \text{ if }0<x\le (2n)^{-\frac{1}{1-\eta}},\\
\left(\frac{x^{\eta+2\ell}}{n}\right)^{\frac{1}{2\ell+1}}& \text{ if } (2n)^{-\frac{1}{1-\eta}}\le x\le 1,\\
n^{-\frac{1}{2\ell+1}}& \text{ if }1\le x\le A.
\end{cases}$$
Then the $L_2$-risk of the estimator $\hat{f}_{h_{0}}$ is bounded by: 
\begin{align*}
	\E{\big[\|{\hat{f}_{h_0}-f}\|_{2}^2\big]}&\lesssim
	\int_0^{(2n)^{-\frac1{1-\eta}}} x^{-2\eta}dx+\int_{(2n)^{-\frac1{1-\eta}}}^{1}  \frac{x^{-\eta}}{nn^{-\frac1{2\ell+1}}x^{\frac{\eta+2\ell}{2\ell+1}}}dx+\int_1^{A}n^{-\frac{2\ell}{2\ell+1}}dx\\
& \lesssim {n^{-\frac{1-2\eta}{1-\eta}}}+{n^{-\frac{1-2\eta}{1-\eta}}}+n^{-\frac{2\ell}{2\ell+1}}. 
\end{align*}
{The estimator converges, possibly very slowly.}
We recover the classical rate when $\eta$ gets close to 0 (See Theorem \ref{thm:UBDP} with $\gamma=\ell,p=2$ and $\beta=0$).
\end{exe} 
	
	\section{Spatial adaptation to inhomogeneous smoothness\label{sec:adapt}}

	In this section, neither the regularity parameters nor the regularity change points are assumed to be known. We adapt the bandwidth selection procedure proposed in Lepski \textit{et al.} (1997) \cite{lepski1997optimal} in a context of density estimation. 
The bandwidth $h$ is chosen from the discrete set \[\mathcal{H}_n=\{ a^{-j}, 0\leq j\leq J\} \quad\text{with}\quad J=\lfloor \log_a(n/\log^3(n)) \rfloor \] 
for $1<a\le 2$. For example, if $a=2$,  the set becomes $\mathcal{H}_n=\{ 1,1/2,\ldots,1/2^J \}$.  The closer $a$ gets to 1 the larger the cardinal of $\mathcal{H}_{n}$ is. Note that the constraint on $J$ ensures that for all $h\in\mathcal H_{n}$ it holds that $nh\ge 1$, which allows to control the variance term. In fact, we impose a slightly stronger condition, $nh\geq \log^3(n)$.  

Following the procedure described in \cite{lepski1997optimal},  we choose the largest  bandwidth  in $\mathcal{H}_{n}$ for which the bias remains less than the variance, or a sharp upper bound of the latter. The variance of the estimator satisfies, for any $h>0$,
\[\var{\hat{f}_h(x)}\leq \frac{1}{nh}\int _{-1}^{1}f(x+uh)K^2(u)du \leq \frac{\| f\|_{\infty}\|K\|_{2}^{2}}{nh}=: v^2(h). \]
If no upper bound for $\| f\|_{\infty}$ is available,  the supremum of $f$ can be estimated, 
a strategy is numerically and theoretically studied in \cite{Lacour}. We also estimate $\norm{f}_{\infty}$ in the numerical Section \ref{sec:Num}. To estimate the bias term, we consider a smaller bandwidth $\eta<h$ and set  
$ B_{h,\eta}(x):=\hat{f}_h(x)-\hat{f}_{\eta}(x).$
It is easy to control the variance of $B_{h,\eta}$ as follows
\begin{align}\label{eq:v2etah}\var{B_{h,\eta}(x)}\leq \frac{\|f\|_{\infty}}{n}\|K_{h}-K_{\eta}\|_{2}^{2}= \frac{\| f\|_{\infty}}{nh}\int_{-1}^{1} \left( K(u)-\frac h{\eta}K\left(\frac{uh}{\eta}\right)\right)^2\rmd u=:v^2(h,\eta).\end{align}
For example, for a rectangular kernel $K(x)=2^{-1}\mathbf1_{[-1,1]}(x)$ one gets 
\[v^2(h,\eta)=\frac{\| f\|_{\infty}}{2nh}\left( \frac \eta h\left(1-\frac h\eta\right)^{2}+1-\frac{\eta}{h}\right)=\frac{\| f\|_{\infty}}{2n}\left( \frac 1\eta-\frac1h\right).\]More generally, it holds that 
\begin{align*}
v^{2}(h,\eta) &\le {2\|f\|_{\infty}\|K\|_{2}^{2}}\left(\frac1{nh}+\frac1{n\eta} \right)=2v^2(h)+2v^2(\eta).
\end{align*}

We select $\hat{h}(x)$ as  the largest bandwidth in $\mathcal H_{n}$ such that, for any $\eta<\hat{h}(x)$, the estimated bias $B_{\hat{h}(x),\eta}(x)$ is smaller than its standard deviation  $v(h,\eta)$ multiplied by a logarithmic factor
\begin{align}\label{eq:hadapt}\hat{h}_{n}(x)=\sup\{h\in\mathcal{H}_{n}: \ \  |\hat{f}_h(x)-\hat{f}_{\eta}(x)|\leq \psi(h,\eta),\ \forall \eta<h,\ \eta\in\mathcal{H}_n\}\end{align}
where $\psi(h,\eta)=2D_1v(h)\lambda(h)+v(h,\eta)\lambda(\eta)$
and $\lambda(h)=1\vee\sqrt{D_2\log(1/h)}$, with $D_1\geq 1$ and  $D_2>4p$. 
The aim is to compare the risk of the estimator $\hat{f}_{\hat{h}(x)}$ to the risk of an oracle estimator  $\hat{f}_{h_0(x)}$ where for any $x$, $h_0(x)$ minimizes the point-wise risk. The optimal bandwidth $h_0$ is difficult to handle because it does not belong to a finite set. Instead, we define a pseudo-oracle $h_n^*$ such that 
\begin{align}\label{eq:hor}h^{*}_{n}(x)= \sup \left\{ h\in\mathcal{H}_n, |f_{\eta}(x)-f(x)|\leq D_1v(h)\lambda(h),\ \forall \eta\leq h,\ \eta\in\mathcal{H}_{n}\right\}.\end{align}
The bias of $f_{h^{*}_{n}(x)}(x)$ is smaller than its variance (up to a logarithmic factor). This property is also satisfied by the estimator $\hat f_{h_{0}}$ for $h_{0}$ as in Theorems \ref{thm:UBPH} and \ref{thm:UBDP} (see \eqref{eq:lienbiaisvariance} and \eqref{eq:lien_biais_variance_Dw}).

\begin{thm}\label{thme:UBadapPoint}
	Consider a density function $f$ bounded by $M$ on $I$, and set $v^2(h)=M\norm{K}_2^2/(nh)$. 
	For $n$ large enough,   for all $x\in I_{0}$, it holds that 	\[\E[{|\hat{f}_{\hat h_{n}(x)}(x)-f(x)|^{p}}]\leq C (\log n)^{\frac p2}v^{p}(h^{*}_{n}(x)),\] where $C$ does not depends on $n$ and $h$.
\end{thm}  

From this result we derive the following Corollary on the class of regularities considered above. Theorem \ref{thme:UBadapPoint} can be used to derive adaptive rates on regularity classes different from the ones introduced here.

\begin{cor}[Rates of convergence]
	\label{cor}
If $f$ belongs to $\Sigma_{\pw}(\alpha,L, M, I)\cap \Sigma(\beta,L, M,I)$, then, for any $p\geq 1$,  if $\rk\geq \alpha$ it holds that
\[\E\big[{\|{\hat{f}_{\hat{h}_n}-f}\|_{p,I_{0}}^p\big]}\leq C (\log n)^{\frac p2}\left(n^{-\frac{1+p\beta}{1+2\beta}}+n^{- \frac{ p\alpha}{2\alpha+1}}\right).\]
If $f$ belongs to $D_{\pw}(\ell,\gamma, L, M, I)\cap\Sigma(\beta,L,M,I)$, with $\rk\geq \ell$, then 
if $\gamma\leq \ell-\beta$, for any $p\geq 1$, then the rate of convergence is \[\E\big[{\|{\hat{f}_{\hat{h}_n}-f}\|_{p,I_{0}}^p\big]}\leq C  (\log n)^{\frac p2}\left(n^{-\frac{1+p\beta}{1+2\beta}}+\log(n)n^{-p \frac{\ell}{2\ell+1}}\right),\]
and if $\gamma>\ell-\beta$, also for $p\geq 1$ 
\[\E\big[{\|{\hat{f}_{\hat{h}_n}-f}\|_{p,I_{0}}^p\big]}\leq C  (\log n)^{\frac p2+1}\left(n^{-\gamma\left(\frac{p\beta}{1+2\beta}+\frac{(\ell-\beta)/\gamma}{2\beta+1}\right)}+n^{-p \frac{\ell}{2\ell+1}}\right).\]
\end{cor}

\section{Numerical results\label{sec:Num}}
We illustrate the previous adaptive method for estimating densities of variable regularity. The program is run on the \texttt{R} software. 
We work with the kernel $K(x)=(\frac98-\frac{15}8 x^2)\mathbf 1_{[-1,1]}(x)$, which has order 4.
The constants $D_1$ and $D_2$ are fixed after a preliminary study at $D_1=1$ and $D_2=0.4$. The estimate of the infinite norm of $f$ in the adaptive procedure is obtained by considering the maximum value of a kernel estimator of $f$ with an intermediate fixed window $1/\sqrt{n}$ (see also \cite{Lacour}) and the result is then multiplied by $1+1/\log n$.
The examples considered are inspired by the examples studied in \cite{DGL}:
\begin{enumerate}[label=(\roman*)]

\item Standard Gaussian distribution with density $f(x)=\frac{1}{\sqrt{2\pi}}e^{-x^2/2}$,
\item
 Laplace   distribution with density $f(x)=\frac{1}{2}e^{-|x|/2}$,
\item Exponential distribution with density $f(x)=e^{-|x|}$,
\item
Uniform distribution with density $f(x)=\units{[-1,1]}(x)$,
\item
Beta distribution with parameters $(1;1.1)$ with density  $f(x)=\frac{10}{11}(1-x)^{0.1}\units{[0,1]}(x)$,
\item
Mixture of Gaussian densities where {\small $$f(x)=\frac{1}{\sqrt{2\pi}}\left( \tfrac12e^{-\frac{x^2}2}+\tfrac{1}{10}e^{-50(x+1)^2}+\tfrac{1}{10}e^{-50(x+\frac12)^2}+\tfrac{1}{10}e^{-50x^2}+\tfrac{1}{10}e^{-50(x-\frac12)^2}+\tfrac{1}{10}e^{-50(x+1)^2}\right),$$} 
\item
Mixture of exponential functions with density  $f(x)=\frac{1}{2}e^{x}\units{x<0}+5e^{-10x}\units{x>0}$,
\item  Truncated Gaussian density $f(x)=\frac{1}{\sqrt{2\pi}} e^{-\frac{x^2}2}\units{x<0}+\frac{1}{2\sqrt{2\pi}} e^{-\frac{x^2}4}\units{x>0}.$
\end{enumerate}
Examples (i) and (vi) are $C^{\infty}$ on $\R$, (iii), (iv), (vii) and (viii) have discontinuity points and are $C^{\infty}$ elsewhere, (ii) is continuous but not differentiable at 0, it is $C^{\infty}$ elsewhere and (v) combine a discontinuity point 0 and a point where it is continuous but non differentiable 1.

In Figure \ref{Fig} we display some graphs of the estimated density with the adaptive procedure $\hat f_{\hat h}$ and the associated estimated bandwidth $\hat h$. Interestingly the estimated bandwidth matches the theoretical shape derived in Theorems \ref{thm:UBPH} and \ref{thm:UBDP}. We compare our estimator to a standard kernel estimator with constant bandwidth selected following the naive Scott rule as $1,06 \hat \sigma_{X}n^{-1/5}$, where $\hat\sigma_{X}$ denotes the square root of the empirical variance of $X$, and computed thanks to the \verb'R' function \verb'density'. This estimator, that we denote by $\hat f_{n^{-1/5}}$ in the sequel, has the advantage to be easily computed and  performs very well in practice. We observe that around the irregularity points our estimator better captures the irregularities. Outside the neighborhood of the irregularities both estimators present similar behaviors.

\begin{figure}
\begin{center}
\begin{tabular}{cc}
(ii) & (iv)\\
\includegraphics[scale=0.35]{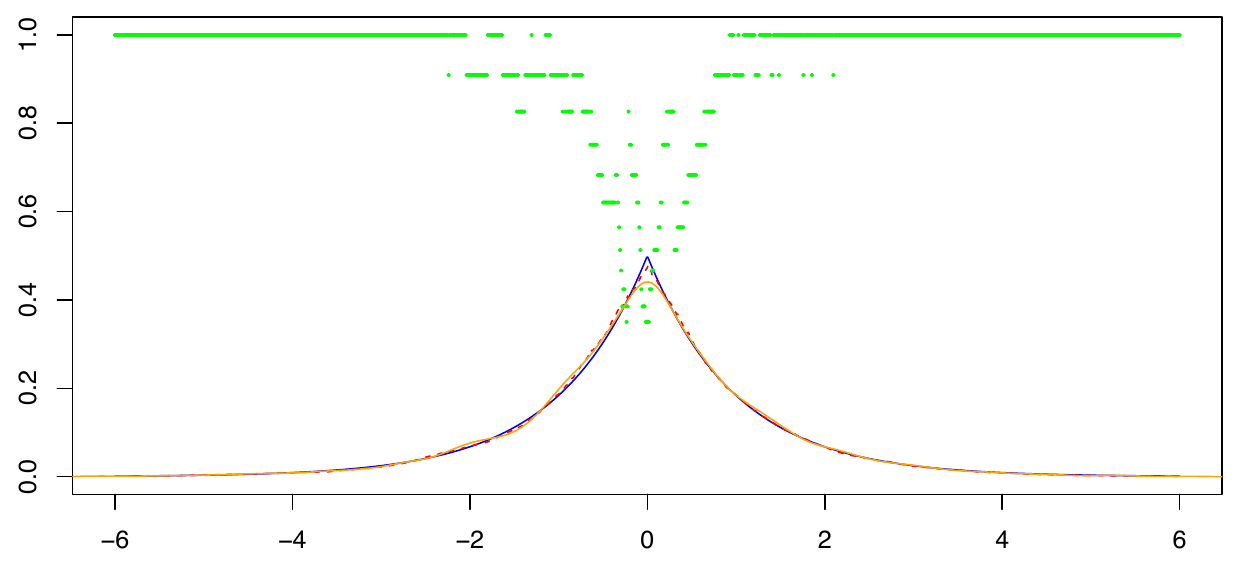}&\includegraphics[scale=0.35]{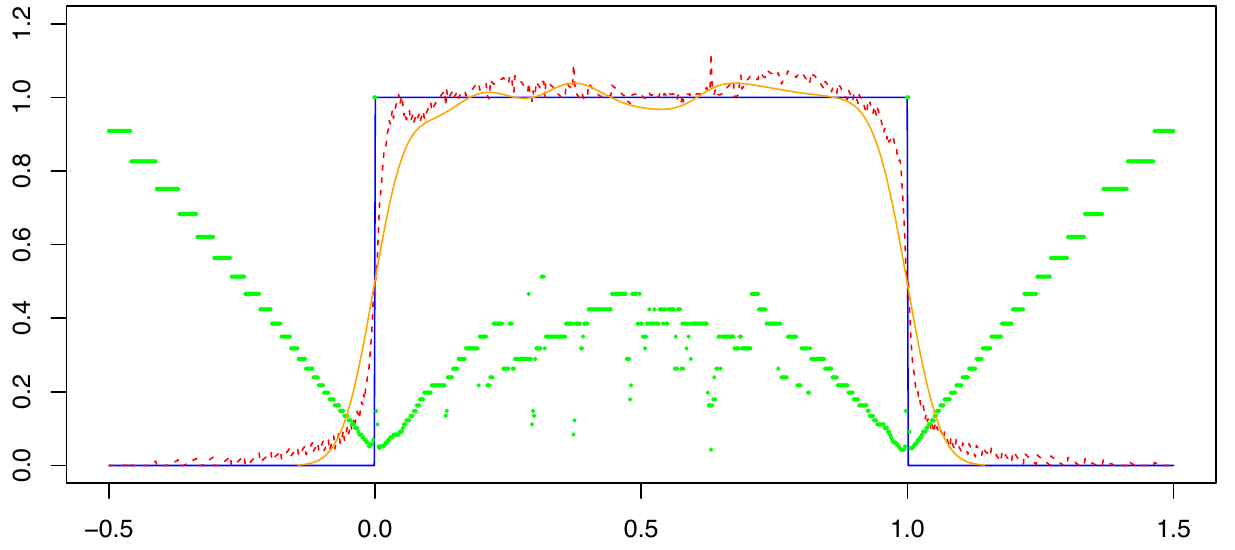}\\

(v) & (vi)\\
\includegraphics[scale=0.35]{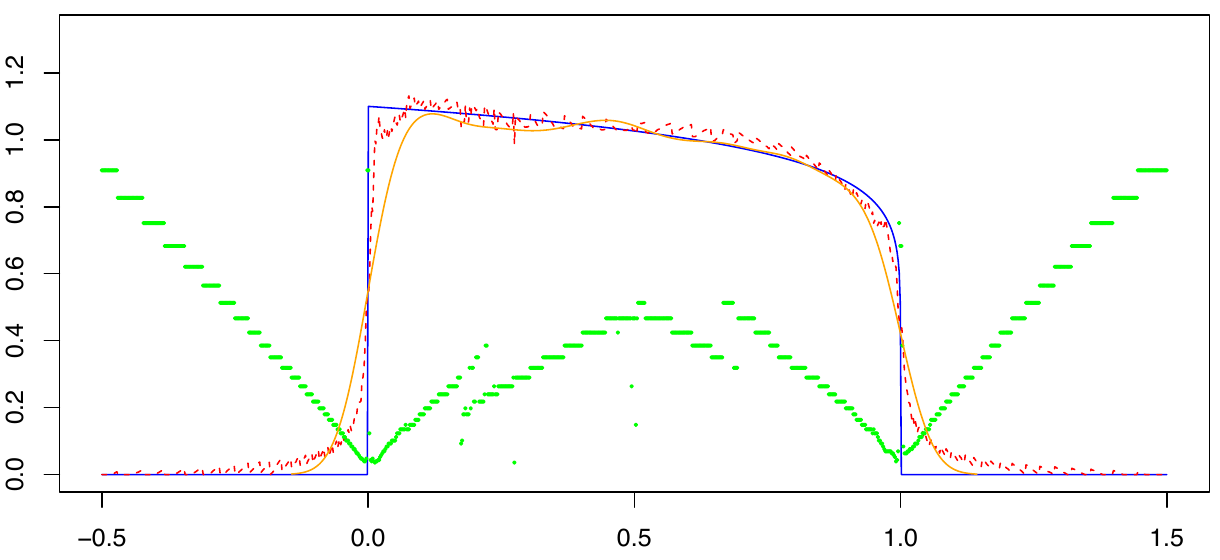}&\includegraphics[scale=0.35]{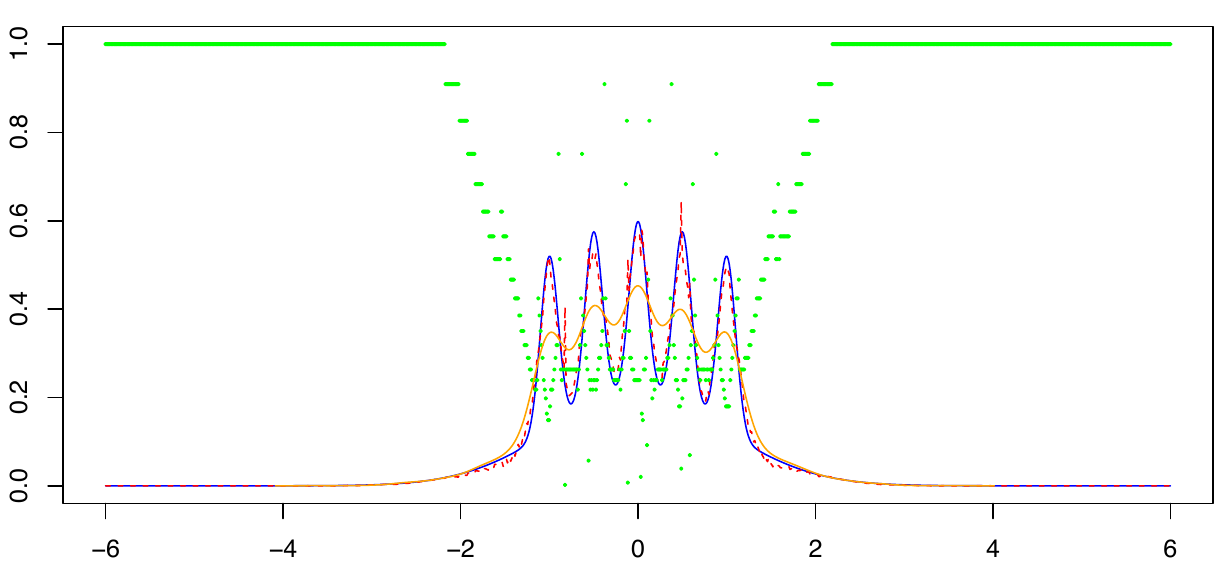}\\
(vii) & (viii)\\
\includegraphics[scale=0.35]{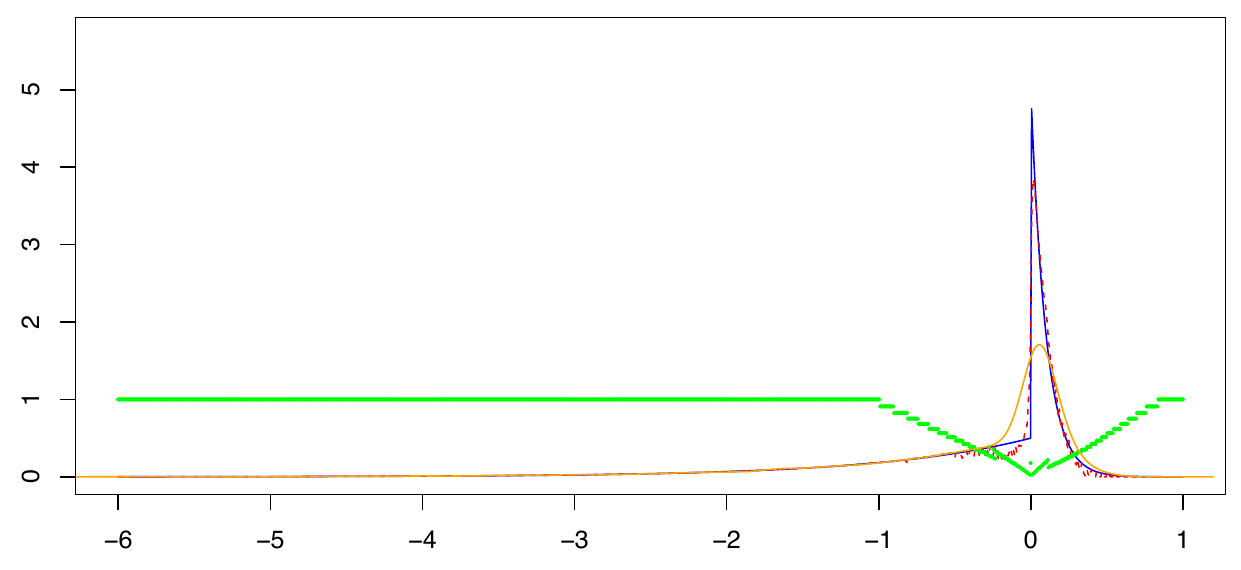}&\includegraphics[scale=0.35]{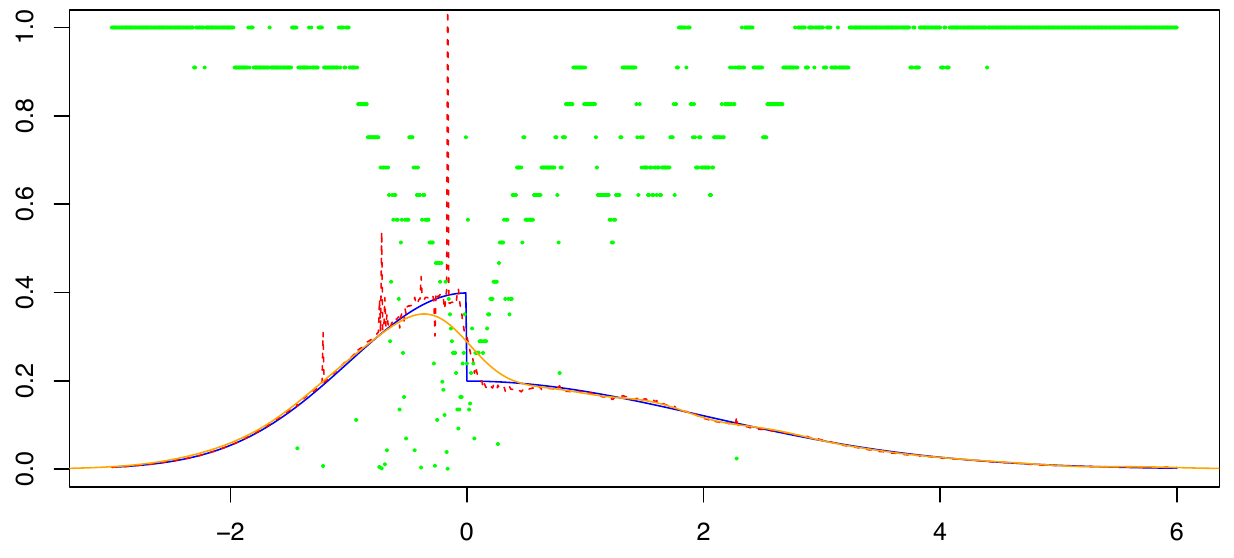}\\
\end{tabular}

\end{center}

\caption{{\footnotesize \label{Fig} For $n=10^4$ observations, illustration of the true density (in blue), our adaptive estimator (in red) and the associated bandwidth (in green) and the fixed bandwidth estimator (in orange). }}
\end{figure}

Consider for an estimator $\hat f$ the normalized $L_{2}$-risks, that allows the numerical comparaison of the different examples, 
$$\frac{ \E[\int_{I_{0}}(\hat f-f)^{2}]}{\int_{I_{0}}f^{2}},$$ where $I_{0}$ denotes the estimation interval. We  evaluate this risk empirically by a Monte Carlo procedure as well as  the associated  deviations (in parenthesis), from 200 independent datasets with different values of sample size. The results are displayed in the following Table \ref{Tab}.

\begin{table}[h!]\begin{center}
\begin{tabular}{cc}

(i), $I_{0}=[-2,2]$ & (ii), $I_{0}=[-2,2]$\\

\begin{tabular}{|c|cc|cc|}
	\hline
	$n$&\multicolumn{2}{c|}{Risk of $\hat{f}_{\hat{h}}$}&\multicolumn{2}{c|}{Risk of $\hat{f}_{n^{-1/5}}$}\\
	\hline
	500&0.0058&{\small (0.0036)}&0.0056&{\small (0.0034)}\\
	1000&0.0027&{\small (0.0021)}&0.0032&{\small (0.0022)}\\
	2000&0.0015&{\small (0.00088)}&0.0020&{\small (0.0010)}\\
	5000&0.00061&{\small (0.00039)}&0.00099&{\small (0.00051)}\\
	10000&0.00036&{\small (0.00024)}&0.00062&{\small (0.00030)}\\
	\hline
\end{tabular}
&\begin{tabular}{|cc|cc|}
	\hline
	\multicolumn{2}{|c|}{Risk of $\hat{f}_{\hat{h}}$}&\multicolumn{2}{c|}{Risk of $\hat{f}_{n^{-1/5}}$}\\

	\hline
0.0101& {\small (0.0046)}& 0.0103& {\small (0.0062)}\\
 0.0075& {\small (0.0029)}& 0.0082& {\small (0.0043)}\\
 0.0056& {\small (0.0019)}& 0.0058&  {\small (0.0025)}\\
 0.0037& {\small (0.0013)}& 0.0035& {\small (0.0013)}\\
 0.0025& {\small (0.00087)}& 0.0022& {\small (0.00080)}\\
	\hline
\end{tabular}\\
&\\

(iii), $I_{0}=[-1,1]$ & (v), $I_{0}=[-1/2,3/2]$\\
\begin{tabular}{|c|cc|cc|}
	\hline
	$n$&\multicolumn{2}{c|}{Risk of $\hat{f}_{\hat{h}}$}&\multicolumn{2}{c|}{Risk of $\hat{f}_{n^{-1/5}}$}\\
\hline
500 &  0.13& {\small (0.0090)}& 0.13&{\small (0.012)}\\
1000 & 0.095& {\small (0.0070)}& 0.12& {\small (0.0077)}\\
2000 & 0.068& {\small (0.0056)}& 0.10& {\small (0.0051)}\\
5000 & 0.039& {\small (0.0028)}& 0.085& {\small (0.0031)}\\
10000& 0.025&{\small (0.0017)}& 0.074&{\small (0.0017)}\\
	\hline
\end{tabular}
&
\begin{tabular}{|cc|cc|}
	\hline
\multicolumn{2}{|c|}{Risk of $\hat{f}_{\hat{h}}$}&\multicolumn{2}{c|}{Risk of $\hat{f}_{n^{-1/5}}$}\\

	\hline
	0.099&{\small (0.0036)}&0.040&{\small (0.0035)}\\
	0.073&{\small (0.0032)}&0.033&{\small (0.0023)}\\
	0.051&{\small (0.0025)}&0.028&{\small (0.0014)}\\
	0.030&{\small (0.0013)}&0.022&{\small (0.00080)}\\
	0.018&{\small (0.00078)}&0.019&{\small (0.00050)}\\
	\hline
\end{tabular}

\\

&\\
\multicolumn{2}{c}{ (vi), $I_{0}=[-2,2]$ }\\

\multicolumn{2}{c}{ \begin{tabular}{|c|cc|cc|}
	\hline
	$n$&\multicolumn{2}{c|}{Risk of $\hat{f}_{\hat{h}}$}&\multicolumn{2}{c|}{Risk of $\hat{f}_{n^{-1/5}}$}\\

	\hline
	500  & 0.12&{\small (0.0047)}& 0.12&{\small (0.0037)}\\
	1000 & 0.12&{\small (0.0039)}& 0.12&{\small (0.0029)}\\
	2000  &0.078& {\small (0.0077)}& 0.11&{\small (0.0020)}\\
	5000  &0.040& {\small (0.0042)}& 0.090&{\small (0.0017)}\\
	10000 &0.024&{\small (0.0025)}& 0.074& {\small (0.0015)}\\
	\hline
	\end{tabular}}

\end{tabular}\caption{{\footnotesize\label{Tab}Empirically estimated normalized $L_{2}$ risks (associated deviation) for the different examples and the estimators $\hat f_{\hat h}$ and $\hat f_{n^{-1/5}}$. }}
\end{center}
\end{table}

Comparing the risks in Table \ref{Tab}, the results of our estimator and the one with a fixed window remain very close, although our estimator gives better results overall. In the Gaussian case (i) (but also the mixture of Gaussian (vi)) where the density is $C^{\infty}$ the rate predicted by the theory for $\hat f_{\hat h}$ is in $n^{-8/9}$, as the kernel is of order 4, whereas that of $\hat f_{n^{-1/5}}$ is in $n^{-4/5}$ since this window is adapted to a density of regularity 2, when $n=10^{4}$ the ratio of these two rates is of the order of 2 which corresponds well to what is observed here. For the other examples, the theory predicts a larger gap between the performances of the estimators than what is measured here. This may be due to the fact that the fixed-window estimator applied to densities with isolated irregularities, is still well adapted over most of the estimation interval making it robust to isolated irregularities. Indeed, the modifications in the bandwidth suggested in \eqref{eq:h0}, \eqref{eq:hD1} or \eqref{eq:hD2} are local and affect intervals of vanishing length.

\section{Proofs \label{sec:prf}}

\subsection{Proof of Theorem \ref{thm:LB}}

	To establish the result we propose two different constructions: one on the space $\Sigma(\alpha,\ell,M,I)\subset  \Sigma_{\pw}(\alpha,L,M,I)\cap \Sigma(\beta,L,M,I)\subset H_{\alpha,\beta}(\eps)$ ensuring the optimality of the rate $n^{-p\alpha/(2\alpha+1)}$ and another on the space $\Sigma(\beta,L,M,I)\cap \Sigma(\alpha,L,M,I\setminus[-\ep,\ep]) \subset H_{\alpha,\beta}(\eps)$ for some small $\eps>0$ attaining the rate $n^{-(p\beta+1)/(2\beta+1)}.$

\paragraph{Lower bound on $\Sigma(\alpha,L,M,\R)$}
Following Theorem 2.7 in Tsybakov (2009) \cite{Tsyb}, the lower bound is established by considering a decision problem between  an increasing number $N+1$ of competing  densities $f_0, \ldots, f_{{N}}$ contained in $  \Sigma(\alpha,L,M,I)$ for some  interval $I$ and positive constants  $0<\alpha$, $L$ and $M$. The following two conditions must be  satisfied
\begin{itemize}
\item[(i)]   $\| f_{j}- f_{k}\|^p_{p,I_{0}} \geq d n^{-\frac{p\alpha}{2\alpha+1} },\   \forall \, 0\leq j<k\leq N$    for some positive constant $d$,
\item[(ii)]  $\frac{1}{N} \sum_{j=1}^{N}\limits  \chi^{2}(\PP^{\otimes n}_{j}, \PP^{\otimes n}_{0} )  \leq \kappa   N,$  for some $0<\kappa<1/8$,
\end{itemize} where
$\PP_{j}$ denotes the probability measure corresponding to the density $f_{j}$  and 
$\chi^{2}$ is the $\chi$-square divergence. 

Consider $I_{0}=[-1,1]$ and  the density $f_{0}(x)=\frac{1}{\sqrt{2\pi\sigma^2}}e^{-x^2/(2\sigma^2)}$ which belongs to the class $ \Sigma(\alpha,1/\sigma^{\alpha+1}L,\sigma^{-1}/\sqrt{2\pi},\R)$ for any  $\alpha$ where $L:=L_{\alpha}=\sup_{\R} |\varphi^{(\lfloor\alpha\rfloor+1)}(x)|$ where $\varphi$ is the density of the $\rond{N}(0,1)$. 
Note that, for any $(\alpha,L,M)$, there exists $\sigma>0$ such that $f_0$   belongs to $\Sigma(\alpha,L/2,M/2,\R)$.

Let $Q\in\N$ and $\delta,\ \gamma>0$. Let us consider the functions \begin{align*}
f_{{\bm \theta}}(x)&= f_{0}(x)+\delta Q^{-\gamma}\sum_{q=0}^{Q-1}\theta_{q}\psi\left(xQ-q\right),
\end{align*} where ${\bm \theta}\in \{0,1\}^{Q}$, $\psi$ is supported on $[0,1]$ and is such that $\int_{0}^{1}\psi(x)\rmd x=0$ and $\psi$  admits derivatives up to order $\lfloor \alpha\rfloor$ such that $\psi^{(j)}(0)=\psi^{(j)}(1)=0$ for $0\le j\le \lfloor \alpha\rfloor$ and such that $|\psi^{(\lfloor \alpha\rfloor)}(x)-\psi^{(\lfloor \alpha\rfloor)}(x')|\le (L/4)|x-x'|^{\alpha-\lfloor \alpha\rfloor}$ for all $x,x'\in[0,1].$ Note that for all $q$, $\psi\left(.Q-q\right)$ is supported on $\left[\frac{q}{Q},\frac{q+1}{Q}\right]$. 
First we show that $f_{\bm {\theta}}$ defined as above is a density and belongs to the desired class. 
\begin{lemma}\label{lem:BIdens}
Let $\alpha\ge 1$ and ${\bm \theta}\in \{0,1\}^{Q}$. Then $f_{0}$   and $f_{{\bm \theta}}$ are densities on $\R$ and belong to
$\Sigma(\alpha,L,M,\R)$  provided that $Q \ge1$, $\gamma\ge \alpha$ and  \begin{align}\label{eq:delta}\delta\le\min\left( \frac{e^{-1/(2\sigma^2)}}{\norm{\psi}_{\infty}\sqrt{2\pi\sigma^2}},  
\frac{M}{2\norm{\psi}_{\infty}},1
\right).\end{align}
\end{lemma}

\begin{proof}[Proof of Lemma \ref{lem:BIdens}]
Note that $f_{0}$   and $f_{{\bm \theta}}$ integrate to 1 on $\R$ by definition of $\psi$. Moreover for $x\in \R\setminus[0,1]$, $f_{\bm \theta}(x)=f_{0}(x)\ge 0$. For $x\in \left[\frac{q}{Q},\frac{q+1}{Q}\right]$, $0\le q\le Q-1$ it holds that
\[
f_{\bm \theta}(x)= 
\frac{1}{\sqrt{2\pi\sigma^2}}
e^{-x^2/(2\sigma^2)}
+\delta Q^{-\gamma }\theta_{q}\psi(xQ-q)
\ge  \frac{1}{\sqrt{2\pi\sigma^{2}}}e^{-1/(2\sigma^2)}-\delta \|\psi\|_{\infty} \ge 0
\]
for $\delta$ as in \eqref{eq:delta}. Note that $\|\psi\|_{\infty}<\infty$ as $\psi$ is continuous on the compact interval $[0,1]$, implying that $\|f_{\bm  \theta}\|_{\infty}\le \frac M2+\delta\|\psi\|_{\infty}\le M$ for $\delta$ as in \eqref{eq:delta}. To check that  it belongs to $\Sigma(\alpha, L,M,\R)$, it remains to control  $|f_{\bm  \theta}^{(\lfloor \alpha\rfloor)}(x)-f_{\bm  \theta}^{(\lfloor \alpha\rfloor)}(x')|$ for all $x,x',\ |x-x'|\le 1$ .
First, if $x,x'\notin [0,1]$, we get
\[ |f_{\bm \theta}^{\lfloor\alpha\rfloor}(x)-f_{\theta}^{\lfloor \alpha\rfloor}(x')\vert =\vert f_0^{\lfloor \alpha\rfloor }(x)-f_0^{\lfloor \alpha\rfloor }(x')\vert \leq \frac{L}{2}|x-x'|^{\alpha-\lfloor \alpha\rfloor }.\] 
If $x\in[0,1]$ and $x'\notin[0,1]$ then there exists $q_{0}$, $x\in[\tfrac {q_{0}}{Q},\tfrac{q_{0}+1}Q]$ and we can write
\begin{align*}|f_{\bm  \theta}^{(\lfloor \alpha\rfloor)}(x)-&f_{\bm  \theta}^{(\lfloor \alpha\rfloor)}(x')|\le \delta Q^{-\gamma+\lfloor \alpha\rfloor}\theta_{q_{0}}|
\psi^{(\lfloor \alpha\rfloor)}(xQ-q_{0})|
+|f_0^{\lfloor \alpha\rfloor }(x)-f_0^{\lfloor \alpha\rfloor }(x')|
\\ &\leq  \delta Q^{-\gamma+\lfloor \alpha\rfloor}\theta_{q_{0}}|
\psi^{(\lfloor \alpha\rfloor)}(xQ-q_{0})-\psi^{(\lfloor \alpha\rfloor)}(0)|
+\frac{L}{2}|x-x'|^{\alpha-\lfloor \alpha\rfloor }\\
&= \delta Q^{-\gamma+\lfloor \alpha\rfloor}\theta_{q_{0}}|
\psi^{(\lfloor \alpha\rfloor)}(xQ-q_{0})-\psi^{(\lfloor \alpha\rfloor)}(1)|
+\frac{L}{2}|x-x'|^{\alpha-\lfloor \alpha\rfloor }\\
&\le \delta\frac{L}{4}Q^{-\gamma+\lfloor \alpha\rfloor}\{(
xQ-q_{0})^{\alpha-\lfloor \alpha\rfloor}\wedge (
1-xQ-q_{0})^{\alpha-\lfloor \alpha\rfloor}\}
+\frac{L}{2}|x-x'|^{\alpha-\lfloor \alpha\rfloor }\\
&\le \frac{L}{4}Q^{-\gamma+\alpha}\{x^{\alpha-\lfloor \alpha\rfloor}\wedge (1-x)^{\alpha-\lfloor \alpha\rfloor}\}
+\frac{L}{2}|x-x'|^{\alpha-\lfloor \alpha\rfloor }\\
&\le \frac{L}{4}Q^{-\gamma+\alpha}|x-x'|^{\alpha-\lfloor \alpha\rfloor}+\frac{L}{2}|x-x'|^{\alpha-\lfloor \alpha\rfloor}. \end{align*} As $\gamma\geq \alpha$, we obtain that $|f_{\bm  \theta}^{(\lfloor \alpha\rfloor)}(x)-f_{\bm  \theta}^{(\lfloor \alpha\rfloor)}(x')|\leq L|x-x'|^{\alpha-\lfloor \alpha\rfloor}$. A similar argument is used if both $x,x'\in[0,1]$.
\end{proof}

Next we show a Lemma that implies point (i) above for $Q$ and $\gamma$ suitably chosen.
\begin{lemma}\label{lem:BInorm} Let $Q\ge 8$, there exist $N\ge 2^{Q/8}$ elements $\{\bm\theta^{(1)},\ldots,\bm\theta^{(N)}\}$ of $\{0,1\}^{Q}$, $\bm\theta^{(0)}=(0,\ldots,0) $, such that it holds
\[\|f_{\bm\theta^{(j)}}-f_{{\bm \theta^{(k)}}}\|_{p,I_{0}}^{p}\ge \frac18\delta^{p}\|\psi\|^{p}Q^{-p\gamma},\quad \forall 0\le j<k\le N.\] \end{lemma}
\begin{proof}[Proof of Lemma \ref{lem:BInorm}]
Let ${\bm \theta},\ {\bm \theta'} \in \{0,1\}^{Q}$, we compute 
\begin{align*}
\|f_{\bm\theta}-f_{{\bm \theta'}}\|_{p,I_{0}}^{p}&=\int_{-1}^{1}\delta^{p}Q^{-p\gamma}\left\vert \sum_{q=0}^{Q-1}(\theta_{q}-\theta'_{q})\psi(xQ-q)\right\vert ^{p}\rmd x\leq \delta^{p}Q^{-p\gamma-1}\|\psi\|^{p}\rho({\bm \theta}, {\bm \theta'} )
\end{align*} where we used that the supports of   the $\psi(.Q-q)$ are disjoint, a change of variable and where $\rho$ denotes the Hamming distance defined by $\rho({\bm \theta}, {\bm \theta'} )=\sum_{q=0}^{Q-1}\mathbf{1}_{\theta_{q}\ne\theta'_{q}}$. Next the Varshamov–Gilbert bound (see Lemma 2.9 in Tsybakov (2009) \cite{Tsyb} ensures that for $Q\ge 8$ there exist $N\ge 2^{Q/8}$ elements $\{\bm\theta^{(1)},\ldots,\bm\theta^{(N)}\}$ of $\{0,1\}^{Q}$ such that $\rho({\bm \theta}^{(j)}, {\bm \theta}^{(k)} )\ge Q/8$ for all $0\le j<k\le N$ with $\bm \theta^{(0)}=(0,\ldots,0)$. This leads to the desired result.
\end{proof}
Finally,  we show a Lemma that enables to establish point (ii) above.\begin{lemma}\label{lem:BIchi} 
 Let $\bm\theta\in\{0,1\}^{Q}$, and $\delta$ satisfying \eqref{eq:delta} and  $\delta\le  {e^{-1/(4\sigma^2)}}/({4\sqrt{\log(2) \sqrt{2\pi\sigma^2}}\norm{\psi}_{2}} ) .$ Then, it holds that 
 $$\chi^{2}(f_{\bm\theta},f_{0})\leq \frac{Q^{-2\gamma}}{16\log(2)}.$$
\end{lemma}
\begin{proof}[Proof of Lemma \ref{lem:BIchi}] Straightforward computations lead to
\begin{align*}
\chi^{2}(f_{\bm\theta},f_{0})&=\int_{\R}\frac{\big(f_{\bm\theta}(x)-f_{0}(x)\big)^{2}}{f_{0}(x)}\rmd x=\delta^{2}Q^{-2\gamma}\sqrt{2\pi\sigma^2}\int_{0}^{1}e^{x^2/(2\sigma^2)}\sum_{q=0}^{Q-1}\theta_{q}\psi^{2}(xQ-q)\rmd x\\
&\le \sqrt{2\pi\sigma^2}e^{1/(2\sigma^2)}\delta^{2}Q^{-2\gamma}\sum_{q=0}^{Q-1}\int_{\frac qQ}^{\frac{q+1}{Q}} \psi^{2}(xQ-q)\rmd x=  \sqrt{2\pi\sigma^2}e^{1/(2\sigma^2)}\delta^{2}\|\psi\|_{2}^{2}Q^{-2\gamma}\\
&\leq \frac{Q^{-2\gamma}}{16\log(2)}.\end{align*}
\end{proof}
For all $1\le j\le N$ it holds,  using that $x\mapsto \log(1+x)\le x$,
\[
	\chi^{2}(\PP_{\bm\theta^{(j)}}^{\otimes n},\PP_{\bm\theta^{(0)}}^{\otimes n})=\left(1+\chi^{2}(f_{\bm\theta^{(j)}},f_{0})\right)^{n}-1\le \exp\left(\frac{nQ^{-2\gamma}}{16\log(2)}\right)\]
If $nQ^{-2\gamma}\leq Q$,  we obtain, as $N\ge 2^{Q/8},\ Q\ge 8$,
\[
\chi^{2}(\PP_{\bm\theta^{(j)}}^{\otimes n},\PP_{\bm\theta^{(0)}}^{\otimes n})= \exp\left(Q/(16\log(2))\right) \le   \kappa N
\] thanks to the condition on $\delta$, for some $\kappa \in(0,\frac18)$. 
Then we need to take $\gamma\geq \alpha $ (to ensure that $f_{\theta}$ belongs to $\Sigma(\alpha,L,M,I))$, and $nQ^{-2\gamma}\leq 1$ (so the $\chi^2$-divergence is small enough). To conclude, we set $\gamma=\alpha$ and  $Q^{-2\alpha}=n^{-p\alpha/(2\alpha+1)}$, with Lemma \ref{lem:BIdens} and \ref{lem:BInorm} together with Theorem 2.7 in Tsybakov (2009) \cite{Tsyb} implies
$$\inf_{\hat f}\sup_{f\in \Sigma(\alpha,L,M,I)}\EE{\|\hat f-f\|^{p}_{p,I_{0}}}\ge cn^{-\frac{p\alpha}{2\alpha+1}},$$ for some positive constant $c$.

\paragraph{Bound on $H_{\alpha,\beta}(\ep)$}

It is sufficient to consider the case where there is only one irregularity at $0$, we thus consider the space $H^{0}_{\alpha,\beta}(\eps)=\Sigma(\beta,L,M,I)\cap \Sigma(\alpha,L,M,I\setminus[-\ep,\ep])$. Contrary to the case studied above, as we wish to recover the rate around the irregularity point 0, we can construct a 2-test hypothesis lower bound as follows.
Consider the same  density $f_{0}$ as before and the same function $\psi$, define the density $f_{1}$  as 
 $f_1(x)=f_0(x)+\delta Q^{-\gamma} \psi(Qx)$
with $\delta$ satisfying \eqref{eq:delta},  $\gamma\geq \beta$, and $Q\geq 1/\eps$. 
 Indeed, it is sufficient to check that the condition $Q\ge 1/\eps$ ensures that $f_{1}$ has regularity $\alpha$ on the domain $I\setminus[0,\tfrac1Q]\subset I\setminus[-\eps,\eps]$ implying that $f_{1}\in H^{0}_{\alpha,\beta}(\eps). $

We have that 
\[
	\|f_{0}-f_{1}\|_{p,I_{0}}^{p}=\int_{0}^{\frac1Q}\delta^{p}Q^{-p\gamma}\psi^p(xQ)\rmd x=\delta^{p}Q^{-p\gamma-1}\|\psi\|_{p}^{p}.
\]
Moreover, 
\begin{align*}
	\chi^{2}(f_{1},f_{0})&=\int_{\R}\frac{\big(f_{1}(x)-f_{0}(x)\big)^{2}}{f_{0}(x)}\rmd x=\delta^{2}Q^{-2\gamma}\sqrt{2\pi\sigma^2}\int_{0}^{\frac1Q}e^{x^2/(2\sigma^2)}\psi^{2}(xQ)\rmd x\\
	&\le C\delta^{2}Q^{-2\gamma}\int_{0}^{\frac{1}{Q}} \psi^{2}(xQ)\rmd x=  C	\delta^{2}\|\psi\|_{\infty}^{2}Q^{-2\gamma-1}.\end{align*}
 Then $\chi^2(f_0^{\otimes n},f_1^{\otimes n})\leq C'nQ^{-2\gamma-1}$. 
Let us set $\gamma=\beta$, and $Q=n^{1/(2\beta+1)}\ge 1/\eps$. By Proposition 2.1 and 2.2 in \cite{Tsyb}  together with the latter lower bound, we obtain for some positive constant $c$: 
\[\inf_{\hat f}\sup_{f\in H^{0}_{\alpha,\beta}(\eps)}\E{\|\hat f-f\|^{p}_{p,I_{0}}}\ge c\left(n^{-\frac{p\alpha}{2\alpha+1}}+n^{-\frac{p\beta+1}{2\beta+1}}\right).\]

\subsection{Proof of Theorem \ref{thm:UBDP}}
	As in the proof of Theorem \ref{thm:UBPH}, the proof is done when there is only one isolated irregularity, at 0 then $c_{0}=1$ and  $0\in I_{0}$. Consider a function $f\in C^{\ell}(\R\setminus\{0\})$. 
By Equation \eqref{eq:Rosprf1}, the variance term $\E{\left[|\hat{f}_{h_0(x)}(x)-f_{h_0(x)}(x)|^p\right]}$ is bounded by  $c_pv(h_0)^p$ for all $p\ge 1$.
In the sequel, we  bound the bias term depending on the value of $x$. 

$\bullet$ If $|x|\leq n^{-1/(2\beta+1)}$,
 as $f$ is in $\Sigma(\beta,L,M,I)$, it always holds that $|f_{h}(x)-f(x)|\leq Lh^{\beta}\norm{K}_1$ (see Equation \eqref{eq:BiasPH1}). As $h_0(x)=(\cons n)^{-1/(2\beta+1)}/2$, observe that  $Lh_0(x)^{-\beta}\norm{K}_1\le v(h_0(x))$, implying that $|f_{h_0(x)}(x)-f(x)|\leq v(h_0(x))$.  The same results, replacing $\beta$ with $\alpha$, hold if $|x|\ge c_{0}$.
  
 $\bullet$ If $c_{0}\geq |x|\geq n^{-1/(2\beta+1)}$, we are in a regime sufficiently away from the irregularity point 0, as the derivative is unbounded it may still affect the evaluation of the bias. For $h$ is such that $0\le h< |x|/2$, then   the $\ell$-derivative $f^{(\ell)}$ is well defined on $[x-h, x+h]$ and smaller  than $ L(|x|/2)^{-\gamma}$. 
  Using   \eqref{eq:BiasTayl} we obtain
 \begin{align}
|f_{h}(x)-f(x)|\leq \frac{\|Ku^{\ell}\|_{1}}{\ell !}h^{\ell}L(|x|-h)^{-\gamma} \leq 2^{\gamma}\norm{K}_1 L h^{\ell} \vert x\vert^{-\gamma}.\label{eq:BIASup}
 \end{align}  The following bound on the bias term holds
  \begin{align}
|f_{h}(x)-f(x)|\leq 2^{\gamma}\norm{K}_1 L\  \min\left\{h^{\ell} \vert x\vert^{-\gamma},h^{\beta}\right\}.\label{eq:BIASup}
 \end{align} The 
 minimum is attained by the first term if $|x|\ge h^{(\ell-\beta)/\gamma} $ and the bandwidth performing the bias-variance compromise is given by $ h_{\gamma}(x)=(\cons n)^{-1/(2\ell+1)}|x|^{2\gamma/(2\ell+1)}$.

In case $\ell-\beta< \gamma$ note that the bias is in $h _{\gamma}^{\ell}(x)|x|^{-\gamma}$, if both the constraint $|x|\ge h_{\gamma}(x)^{(\ell-\beta)/\gamma} $ and $h_{\gamma}(x)\le |x|/2$ are satisfied. Noting that $h_{\gamma}(x)\le |x|^{{\gamma}/(\ell-\beta)}\le |x|$ as $|x|\le 1$ and $\gamma>\ell-\beta$ implies that one should select $h_{0}$ of the order of $h_{\gamma}(x)$ for all $x$ such that $|x|\ge x_{0}:=(\cons n)^{-(\ell-\beta)/(\gamma(2\beta+1))}$, where the latter value is computed such that $h_{\gamma}(x_{0})=x_{0}^{{\gamma}/(\ell-\beta)}$, and as $n^{-1/(2\beta+1)}$ otherwise. Therefore, we set $h_{0}(x)=h_{\gamma}(x)/2$ for all $|x|\in\left[x_{0},c_{0}\right]$ and $h_{0}(x)=(\cons n)^{-1/(2\beta+1)}/2$ for $|x|\in\left[(\cons n)^{-1/(2\beta+1)},x_{0}\right]$.

In case $\ell-\beta\ge \gamma$, in comparison with the later case one can have $|x|\le h_{\gamma}(x)\le |x|^{-\gamma/(\ell-\beta)}$, this occurs on the interval $|x|\le x_{1}$ where $x_{1}:=(\cons n)^{-1/(2(\ell-\gamma)+1)}$ is such that $x_{1}=h_{\gamma}(x_{1}).$ Note that using that $\ell-\beta\ge \gamma$ we get that $x_{1}\ge x_{0}$ and that $x_{0}\le (\cons n)^{-1/(2\beta+1)}$.  Therefore, we set  $h_{0}(x)=|x|/2$ for all $|x|\in\left[(\cons n)^{-1/(2\beta+1)},x_{1}\right]$ and $h_{0}(x)=h_{\gamma}(x)/2$ for all $|x|\in\left[x_{1},c_{0}\right]$.

 In all above cases, $h_0$ is selected as half of the value that minimizes the bias-variance compromise. Since the bias is an increasing function of $h$ and  $v(h)$ a decreasing function, we maintain the property that $|f_{h_0}(x)-f(x)|\leq v(h_0(x))$.

Then, for any $x\in I_0$, we have that 
	\begin{equation}
		|f_{h_0(x)}(x)-f(x)|\leq v(h_0(x)). \label{eq:lien_biais_variance_Dw}
	\end{equation}

It remains to compute the risk. 
Let us first consider the case when $\gamma\leq \ell-\beta$. We get for a constant $C$ whose value may change from line to line 
\begin{align*}
	\EE{ \int_{I_{0}} |\hat{f}_{h_0(x)}(x)-f(x)|^p\rmd x}&\le C\bigg(\int_0^{{(\kappa n)}^{-\frac1{2\beta+1}}} n^{-\frac{p\beta}{2\beta+1}} \rmd x + \int_{{(\kappa n)}^{-\frac1{2\beta+1}}}^{{(\kappa n)}^{-\frac1{2\ell-2\gamma+1}}} \frac{1}{(nx)^{\frac p2}}\rmd x\\
		&+ \int_{{(\kappa n)}^{-\frac1{2\ell-2\gamma+1}}}^1 n^{-\frac p2} n^{\frac{p/2}{2\ell+1}} x^{-\frac{\gamma p}{2\ell+1}}\rmd x+\int_{I_0\setminus[-1,1]} n^{-{\frac p2}}n^{\frac{p/2}{2\ell+1}}\rmd x\bigg) \\
		 &\le C\bigg( n^{-\frac{p\beta+1}{2\beta+1}} +  \frac{\log(n)}{n}\units{p=2}{+ n^{-\frac{p(l-\gamma)+1}{2(l-\gamma)+1}} \units{1\le p<2}}\\
		&+ n^{- \frac{p\ell}{2\ell+1}}\left( 1+ \log(n)\units{ \gamma p=2\ell+1}+n^{\frac{1}{1+2\ell-2\gamma}\left(\frac{\gamma p}{2\ell+1}-1\right)}\units{\gamma p>2\ell+1}\right)\bigg).
\end{align*}
As $\gamma\leq \ell-\beta$, it follows that $1+2\ell-2\gamma\geq 1+2\beta$ and
\[
	-\frac{p\ell}{2\ell+1}+\frac{1}{1+2\ell-2\gamma}\left(\frac{\gamma p}{2\ell+1}-1\right)\leq -\frac{p\ell}{2\ell+1}+\frac{1}{1+2\beta}\left(\frac{(\ell-\beta) p}{2\ell+1}-1\right)
	=\frac{-p\beta-1}{2\beta+1}.
	\]
	Moreover, if $1\le p<2$, the function $x\to (px+1)/(2x+1)$ is decreasing, implying that $$n^{-(p(\ell-\gamma)+1)/(2(\ell-\gamma)+1)}\leq n^{-(p\ell+1)/(2l+1)}\leq n^{-p\ell/(2\ell+1)}.$$ 
Therefore, if $\gamma\leq \ell-\beta$, we get the desired bound in
\[	\E\left[ \|\hat{f}_{h_0}-f\|^p_{p,I_{0}}\right]\lesssim n^{-\frac{p\beta+1}{2\beta+1}} +  \frac{\log(n)}{n}\units{p=2} 
+ n^{- \frac{p\ell}{2\ell+1}}\left( 1+ \log(n)\units{\gamma p=2\ell+1}\right).\]
Otherwise if $\gamma>\ell-\beta$, the proof is done in the same way, and we get 
\begin{align*}
	\E\left[ \|\hat{f}_{h_0}-f\|^p_{p,I_{0}}\right]&\lesssim n^{-\frac{\gamma p\beta+(\ell-\beta)}{\gamma(2\beta+1)}} + n^{- \frac{p\ell}{2\ell+1}}\left( 1+ \log(n)\units{\gamma p=2\ell+1}+n^{\frac{\ell-\beta}{\gamma(1+2\beta)}\left(\frac{\gamma p}{2\ell+1}-1\right)}\units{\gamma p>2\ell+1}\right).
\end{align*}
To complete the proof note that
\[ \frac{\gamma p\beta+(\ell-\beta)}{\gamma(2\beta+1)}= \frac{p\ell}{2\ell+1}-\frac{\ell-\beta}{\gamma(1+2\beta)}\left(\frac{\gamma p}{2\ell+1}-1\right).\]
Generalization to $m$ isolated regularities in $I_{0}$ is straightforward by cutting the integral around each irregularity point $x_{i},\ 1\le i\le m$.

\subsection{Proof of Theorem \ref{thme:UBadapPoint} }
	Fix $x\in I_0$. The point-wise risk is bounded on the two  events $\{\hat h_{n}(x)\ge h^{*}_{n}(x)\}$ and its complementary $\{\hat h_{n}(x)< h^{*}_{n}(x)\}$. To clarify the reading we drop the dependency in $x$.
	
On the event  $\{\hat h_{n}\ge h^{*}_{n}\}$ the variance term can be easily controlled compared to the bias term. For some generic constant $C$, whose value may change from line to line, we write
	\begin{align*}
	\E[{|\hat{f}_{\hat h_{n}}(x)-f(x)|^{p}\mathbf{1}_{\{\hat h_{n}\ge h^{*}_{n}\}}]}\le C \Big(&	\E[{|\hat{f}_{\hat h_{n}}(x)-\hat f_{h^{*}_{n}}(x)|^{p}\mathbf{1}_{\{\hat h_{n}\ge h^{*}_{n}\}}]}+
	\E[{|\hat{f}_{ h^{*}_{n}}(x)- f_{h^{*}_{n}}(x)|^{p}]}\\ &+
	\E[{|{f}_{ h^{*}_{n}}(x)- f(x)|^{p}]}\Big)=:C(T_{1}+T_{2}+T_{3}).
	\end{align*} On $\{\hat h_{n}\ge h^{*}_{n}\}$ using the definition of $\hat{h}_n$ \eqref{eq:hadapt}, it holds that 
\[|\hat{f}_{\hat h_{n}}(x)-\hat f_{h^{*}_{n}}(x)|\le \psi(\hat h_{n},h^{*}_{n})\le \sup_{h\in\mathcal H_{n}, h\ge h^{*}_{n}}\psi(h,h^{*}_{n})\le 2(D_{1}+2)v(h^{*}_{n})\lambda(h^{*}_{n}).\]
Indeed, 
$\psi(h,h')=2D_1v(h)\lambda(h)+v(h,h')\lambda(h')$ and for $h<h'$, 
\[v^2(h,h')\leq 2v^2(h)+2v^2(h')\leq 4v^2(h).\]
As $\lambda$ is non increasing, it follows that $T_{1}\le C v^{p}(h^{*}_{n})\lambda^{p}(h^{*}_{n})$.  By the definition of $h_n^*$ in \eqref{eq:hor}, it holds that $T_{3}\le Cv^{p}(h^{*}_{n})\lambda^{p}(h^{*}_{n}).$ The term $T_{2}$ is bounded using Equation
	 \eqref{eq:Rosprf1}, for $p\ge 1$ and using $nh^{*}_{n}\ge 1$, one gets 
	\begin{align}\label{eq:T2} T_{2}\le C  \frac{1}{(nh^{*}_n)^{p/2}}\le C v^{p}(h^{*}_{n}).\end{align}
	 Gathering all terms we get the desired result on $\{\hat h_{n}\ge h^{*}_{n}\}$.\\

	The study on the event $\{\hat h_{n}< h^{*}_{n}\}$ is more difficult as we no longer have  control on the variance of the estimator, however, this event occurs with low probability.
Consider for $h>\eta$, $(h,\eta)\in\mathcal H_{n}$	the event $$A_{n}(x,h,\eta):=\{|\hat{f}_{h}(x)-\hat{f}_{\eta}(x)|\geq \psi(h,\eta)\}$$ and note that $\{\hat h_{n}=ha^{-1}\}\subset\bigcup_{\eta\in\mathcal H_{n},\,\eta<h}A_{n}(x,h,\eta)$. Therefore it holds that
$$ \{\hat h_{n}< h^{*}_{n}\}\subset \bigcup_{h\in \mathcal H_{n},\, h\leq h^{*}_{n}}\bigcup_{\eta\in\mathcal H_{n},\,\eta<h}A_{n}(x,h,\eta)$$ and the risk can be bounded as follows
\begin{align*}
\E[{|\hat{f}_{\hat h_{n}}(x)-f(x)|^{p}\mathbf{1}_{\{\hat h_{n}< h^{*}_{n}\}}]}\le C &\sum_{h\in \mathcal H_{n},\, h<ah^{*}_{n}}\ \ \sum_{\eta\in\mathcal H_{n},\,\eta<h}\E[|\hat{f}_{ h}(x)-f(x)|^{p}\mathbf{1}_{A_{n}(x,h,\eta)}].
\end{align*} Equation \eqref{eq:hor}, for any $h<h^{*}_{n}$, yields $|{f}_{ h}(x)-f(x)|\le  D_{1}v(h^{*}_{n})\lambda(h_{n}^{*})\le D_{1} v(h)\lambda(h).$  
It follows using the latter together with \eqref{eq:T2} and the Cauchy-Schwarz inequality
\begin{align*}
\E[{|\hat{f}_{\hat h_{n}}(x)-f(x)|^{p}\mathbf{1}_{\{\hat h_{n}< h^{*}_{n}\}}]}\le C &\sum_{h\in \mathcal H_{n},\, h<ah^{*}_{n}}\ v^{p}(h)\lambda^{p}(h)\sum_{\eta\in\mathcal H_{n},\,\eta<h}\sqrt{\PP\big(A_{n}(x,h,\eta)\big)}\\
\le C &n^{-p/2}\sum_{h\in \mathcal H_{n},\, h<ah^{*}_{n}}\lambda^p(h)h^{-p/2}\sum_{\eta\in\mathcal H_{n},\,\eta<h}\sqrt{\PP\big(A_{n}(x,h,\eta)\big)}.
\end{align*}
We now control the probability of the event $A_{n}(x,h,\eta)$. The term $\hat{f}_{h}(x)-\hat{f}_{\eta}(x)$ appearing in $A_{n}(x,h,\eta)$ can be decomposed into a bias term and a variance term: 
\[\hat f_{h}-\hat{f}_{\eta}=[(\hat{f}_{h}-{f}_h)-(\hat f_{\eta}-f_{\eta})]+[(f_h-f)-(f_{\eta}-f)]=:V+B.\]
Since $\eta$ and $h$ are smaller than $h^{*}_{n}$, the bias term  is easily controlled with \eqref{eq:hor}\[|B|=|f_{h}(x)-f_\eta(x)|\leq 2D_1v(h^{*}_n)\lambda(h^{*}_n)\leq 2D_1v(h)\lambda(h)\]
as the function $h\mapsto  v(h)\lambda(h)$ is decreasing. 
Then, the variance term must be large on $A_{n}(x,h,\eta)$: indeed as  $\psi(h,\eta)=2D_1v(h)\lambda(h)+v(h,\eta)\lambda(\eta)$, it implies that \[|V|=| \hat{f}_{\eta}(x)-\hat{f}_h(x)-f_{\eta}(x)+f_h(x)|> v(h,\eta)\lambda(\eta).\]
Let us apply the Bernstein inequality to this term. Define the centered variables $Z_i=K_{h}(x-X_i) -K_\eta(x-X_i) -(f_{h}(x)-f_\eta(x))$ which satisfy $\var{Z_i}\le nv^{2}(h,\eta)$ (see \eqref{eq:v2etah}). 
Moreover, since $\eta<h$, it holds $|Z_i|\leq \frac{4\norm{K}_{\infty} }{\eta}$. Therefore, we apply the Bernstein inequality and get
\begin{align*}
\PP\left( \left|\frac{1}{n}\sum_{i=1}^n Z_i\right|\geq v(h,\eta)\lambda(\eta)\right)&
 \leq 2\exp\left( -\frac{nv^2(h,\eta)\lambda^2(\eta)/2}{nv^2(h,\eta)+\frac{4v(h,\eta)\lambda(\eta)\norm{K}_{\infty}}{3\eta}}\right) \\ &
\leq 2\exp\left( -\frac{\lambda^2(\eta)/2}{1+\frac{4\lambda(\eta)\norm{K}_{\infty}}{3n\eta v(h,\eta)}}\right).
\end{align*}
The quantity   $v(h,\eta)$ needs to be bounded from below. Remark that the Cauchy-Schwarz inequality gives
\[
v^2(h,\eta)= \frac{\norm{f}_{\infty}}{n}\left( \frac{\| K\|_{2}^2}{h}+\frac{\norm{K}_2^2}{\eta}-2\langle K_{h},K_{\eta}\rangle \right)
\geq \frac{\norm{f}_{\infty}}{n\eta}\norm{K}_{2}^2\left(1+\frac{\eta}{h}-2\sqrt{\frac{\eta}{h}}\right).
\]
The function $x\to 1+x-2\sqrt{x}=(1-\sqrt{x})^2$ is positive, minimum at  $x=1$, and decreasing for $0<x<1$. As $\eta/h\leq 1/a$, we get that $v^2(h,\eta)\geq \frac{\norm{f}_{\infty}}{n\eta}\norm{K}_{2}^2(1+\frac{1}{a}-\frac{2}{\sqrt{a}})$.  Using that $n\eta\ge \log^{3}n$, we derive that
\[ 0\le\frac{4\lambda(\eta)\|K\|_{\infty}}{3n\eta v(h,\eta)}\leq \frac{4\|K\|_{\infty}\sqrt{D_{2}}}{3\|K\|_{2}\sqrt{\|f\|_{\infty}(1-{\frac1a-\frac{2}{\sqrt{a}}})}} \frac{1}{\log n}.\]
This quantity tends to 0 and is smaller than 1 for $n$ such that
\[n\ge \exp\left(\frac{4\|K\|_{\infty}\sqrt{D_{2}}}{3\|K\|_{2}\sqrt{\norm{f}_{\infty}(1-{\frac1a-\frac{2}{\sqrt{a}}})}}\right).\] 
 Therefore, we obtain for $\eta<h< h^{*}_{n}$ that 
\[ \PP(A_{n}(x,h,\eta))\le\PP\left(\left|\frac{1}{n}\sum_{i=1}^n Z_i\right|\geq v(h,\eta)\lambda(\eta)\right)\leq 2\exp(-\lambda^2(\eta)/4)= 2\eta^{\frac{D_2}4}. \]
Gathering all  above results leads to
\begin{align*}
\E[{|\hat{f}_{\hat h_{n}}(x)-f(x)|^{p}\mathbf{1}_{\{\hat h_{n}< h^{*}_{n}\}}]}\le C &n^{-p/2}\sum_{h\in \mathcal H_{n},\, h<ah^{*}_{n}}\ \lambda^p(h) h^{-p/2} \sum_{\eta\in\mathcal H_{n},\,\eta<h}\eta^{\frac{D_{2}}{8}},
\end{align*} 
where using that $a>1$ and $D_{2}>0$  we get
$$ \sum_{\eta\in\mathcal H_{n},\,\eta<h}\eta^{\frac{D_{2}}{8}}=\sum_{j=\lfloor\log_a(1/h)\rfloor+1}^{\lfloor\log_a(n/\log^3(n))\rfloor} a^{-jD_2/8}\leq \frac{h^{D_2/8}}{1-a^{-D_2/8}}.$$
This allows us to write as $D_{2}/8-p/2>0$
\begin{align*}
\E[{|\hat{f}_{\hat h_{n}}(x)-f(x)|^{p}\mathbf{1}_{\{\hat h_{n}< h^{*}_{n}\}}]}\le &C n^{-p/2}\sum_{h\in \mathcal H_{n},\, h<ah^{*}_{n}}\ \lambda^p(h) h^{D_{2}/8-p/2}\\
\le& Cn^{-p/2} (\log n)^{p/2}{h^{*}_{n}}^{D_{2}/8-p/2}\\\le &  C(\log n)^{p/2} v^{p}(h^{*}_{n}){h^{*}_{n}}^{D_{2}/8}\leq  C{(\log n)^{p/2}} v^p(h_n^*), 
\end{align*} as $h_n^*\leq 1$. 
This ensures the desired result.

\subsection{Proof of Corollary \ref{cor}}

By Equations \eqref{eq:lienbiaisvariance} and \eqref{eq:lien_biais_variance_Dw}, we have that, for any $x\in I_0$, 
$ |f_{h_0(x)}(x)-f(x)|\leq v(h_0(x)).$
Since the bandwidth $h_0(x)$ does not necessarily belong to $\mathcal{H}_n$, we define $$h_{0,n}(x)=\max\{h\in\mathcal{H}_n, h\leq h_0(x)\}.$$ Therefore, it holds that $h_{0,n}(x)\geq a^{-1}h_0(x)$, then $v(h_{0,n}(x))\leq av(h_0(x))$.
By construction, for any $h\leq h_{0,n}(x)$, 
\[|f_{h}(x)-f(x)|\leq v(h_{0,n}(x))\leq D_1v(h_{0,n}(x))\lambda(h_{0,n}(x)).\]
So $h_{0,n}(x)\leq h_n^*(x)$ and, as the variance is increasing, 
\[\E\left[\|{\hat{f}_{\hat{h}_n}-f\|}^p_{p,I_0}\right]\leq C(\log(n))^{p/2}\int_{I_0} v^p(h_n^*(x))dx\leq aC(\log(n))^{p/2}\int_{I_0}v^p(h_0(x))dx.\]
As 
\[\int_{I_0} v^p(h_0(x))dx\leq \E\left[\|{\hat{f}_{{h}_0}-f\|}^p_{p,I_0}\right],\]
we obtain the desired result.

\bibliographystyle{apalike}

{\small	
\begin{thebibliography}{99}

\bibitem{abramson1982bandwidth}
Abramson, I. On bandwidth variation in kernel estimates-a square root law. 
 \textit{The Annals of Statistics}, 1217--1223, 1982.




\bibitem{breiman1977variable}
Breiman, L. and Meisel, W. and Purcell, E. Variable kernel estimates of multivariate densities. \textit{Technometrics}, 19, 135--144, 1977.

\bibitem{CFM}
Cheng, M. Y., Fan, J., and Marron, J. S. On automatic boundary corrections. \textit{The Annals of Statistics}, 25(4), 1691-1708, 1997.

\bibitem{CH} Cline D. B. H. and  Hart J. D. Kernel Estimation of Densities with Discontinuities or Discontinuous Derivatives \textit{ Statistics: A Journal of Theoretical and Applied Statistics}, 22:1, 69-84, 1991.

\bibitem{couallier1999estimation}
Couallier, V. Estimation non paramétrique d'une discontinuité dans une densité. \emph{Comptes Rendus de l'Académie des Sciences-Series I-Mathematics}, 633--636, 1999.


\bibitem{DGL}
Desmet, L., Gijbels, I. and Lambert, A. Estimation of irregular probability densities. \textit{Nonparametrics and robustness in modern statistical inference and time series analysis: a Festschrift in honor of Professor Jana Jurečková, 46–61, Inst. Math. Stat. (IMS) Collect., 7, Inst. Math. Statist., Beachwood, OH}, 2010. 




\bibitem{Eeden}
van Eeden, C. Mean integrated squared error of kernel estimators when the density and its derivative are not necessarily continuous. \textit{Ann Inst Stat Math} 37, 461–472, 1985.

\bibitem{van1997note} van Es, B. A note on the integrated squared error of a kernel density estimator in non-smooth cases. \emph{Statistics \& probability letters,} 241--250, 1997.


\bibitem{fan1992variable}
Fan, J. and Gijbels, I. Variable bandwidth and local linear regression smoothers. \textit{The Annals of Statistics,} 2008--2036, 1992.

\bibitem{goldenshluger2008universal}
Goldenshluger, A. and Lepski, O. Universal pointwise selection rule in multivariate function estimation. \textit{Bernoulli}, \textbf{14}, 1150--1190, 2008.

\bibitem{MR3230001} Goldenshluger, A. and Lepski, O. On adaptive minimax density estimation on {$R^d$}. \textit{Probab. Theory Related Fields,} \textbf{159}, 479--543, 2014.

\bibitem{Gosh} Ghosh, B. K., and  Huang W.M.. Optimum bandwidths and kernels for estimating certain discontinuous densities. \emph{Annals of the Institute of Statistical Mathematics}, 44, 563-577, 1992.

\bibitem{HHM}
Hall, P., Hu, T. C., and Marron, J. S.  Improved Variable Window Kernel Estimates of Probability Densities. \textit{The Annals of Statistics,} 23(1), 1–10, 1995.

\bibitem{HM}
Hall, P.,and Marron, J.S. Variable window width kernel estimates of probability densities. \textit{Probab. Th. Rel. Fields.} 80, 37–49, 1988.




\bibitem{10.1214/aos/1176347736}
Hasminskii, R. and Ibragimov, I. 
\newblock {On Density Estimation in the View of Kolmogorov's Ideas in
  Approximation Theory}.
\newblock {\em The Annals of Statistics}, 18(3):999 -- 1010, 1990.


\bibitem{IHRos}
Ibragimov, R. and Sharakhmetov, S. The exact constant in the Rosenthal inequality for random variables with mean zero. \emph{Theory of Probability \& Its Applications}, 46(1), 127-132, 2002.

\bibitem{Ka}Karunamuni, R. J. A new smoothness quantification in kernel density estimation. \emph{Journal of statistical planning and inference} 27.3, 361-373, 1991.


\bibitem{Lacour} Lacour, C.  Adaptive estimation of the transition density of a Markov chain. {\em  Ann. Inst. H. Poincaré Probab. Statist.} {\bf  43}, no. 5, 571-597, 2007. 

\bibitem{lepski1997optimal}
	{Lepski, O. and Mammen, E. and Spokoiny, V. },
		{Optimal spatial adaptation to inhomogeneous smoothness: an approach based on kernel estimates with variable bandwidth selectors},  \textit{The Annals of Statistics}, {929--947}, 
	 {1997}.
	 
\bibitem{schuster1985incorporating}
Schuster, E. {Incorporating support constraints into nonparametric estimators of densities}, \textit{Communications in Statistics-Theory and methods}, {\bf{14}}, {1123--1136}, {1985}.
 	 
	\bibitem{TS}
 Terrell, G. R., and Scott, D. W.. Variable Kernel Density Estimation. \textit{The Annals of Statistics}, 20(3), 1236–1265, 1992. 
	 
	 \bibitem{Tsyb}
   Tsybakov, A. B. \textit{Introduction to nonparametric estimation.} Springer Series in Statistics. Springer, New York, 2009.





\end{thebibliography}}

\end{document}